\documentclass[a4paper,12pt]{article}

\usepackage{amsmath}
\usepackage{amsthm}
\usepackage{amssymb}
\usepackage[latin1]{inputenc}
\usepackage[english]{babel}
\usepackage{tikz}
\usetikzlibrary{shapes,arrows,positioning}
\usepackage{hyperref}

\theoremstyle{plain}
\newtheorem{thm}{Theorem}[section]
\newtheorem{prop}[thm]{Proposition}
\newtheorem{lem}[thm]{Lemma}
\newtheorem{cor}[thm]{Corollary}
\newtheorem{que}[thm]{Question}

\theoremstyle{definition}
\newtheorem{defn}[thm]{Definition}
\theoremstyle{remark}

\newtheorem{rem}[thm]{Remark}

\DeclareMathOperator{\ad}{ad}
\DeclareMathOperator{\Span}{span}

\DeclareMathOperator{\tr}{tr}

\DeclareMathOperator{\rk}{rk}
\DeclareMathOperator{\card}{card}

\DeclareMathOperator{\End}{End}

\title{Almost K\"ahler geometry of adjoint orbits of semisimple Lie groups}
\author{Alberto Della Vedova\footnote{Universit\`a di Milano - Bicocca, alberto.dellavedova@unimib.it}, Alice Gatti\footnote{Universit\`a di Pavia, alice.gatti@unimib.it}}

\begin{document}
\newcommand{\lrar}[1]{\langle #1 \rangle}
\newcommand{\lrbr}[1]{\lbrace #1 \rbrace}
\newcommand{\norm}[1]{\lvert\lvert #1 \rvert\rvert}
\maketitle

\abstract{We study the almost K\"ahler geometry of adjoint orbits of non-compact real semisimple Lie groups endowed with the Kirillov-Kostant-Souriau symplectic form and a canonically defined almost complex structure.
We give explicit formulas for the Chern-Ricci form, the Hermitian scalar curvature and the Nijenhuis tensor in terms of root data. We also discuss when the Chern-Ricci form is a multiple of the symplectic form, and when compact quotients of these orbits are of K\"ahler type.}

\section{Introduction}

The aim of this paper is to study the geometry of a canonically defined almost K\"ahler structure on adjoint orbits of real semisimple Lie groups.
While our approach covers the extensively studied case where the group $G$ is compact (see e.g. \cite[Chapter 8]{Besse1987} and references therein), we focus on the case where $G$ is a non-compact real semisimple Lie group.
Denoted by $\mathfrak g$ the Lie algebra of $G$, our main object of study will be an adjoint orbit $G/V$ of an element $v \in \mathfrak g$ with compact stabilizer $V \subset G$.
By this assumption on the stabilizer, the orbit $G/V$ is never compact.
However the reader interested in compact almost K\"ahler manifolds can eventually get plenty of such manifolds by modding out by the action of some uniform lattice of $G$.

By semisemplicity assumption on $G$, the orbit $G/V$ is canonically isomorphic to a co-adjoint orbit and, as such, it is equipped with the Kirillov-Kostant-Souriau symplectic form $\omega$.
Moreover, among all existing compatible almost complex structures on $(G/V,\omega)$, we consider a homogeneous $J$ that can be canonically defined in terms of root data of $\mathfrak g$ (see \cite[Section 4.2]{DellaVedova2017}, or \cite{AlekseevskyPodesta2018} for an alternate description).
Such a $J$ is independent of $v$, in some sense that will be apparent. 
This makes very easy to show that for any $w \in \mathfrak g$ with stabilizer $W \subset V$, the canonical projection $G/W \to G/V$ is pseudo-holomorphic (see Proposition \ref{prop:pseudoholo}).
This should be compared with the more common Hodge-theoretic `non-classical' situation where one considers $G/W$ and $G/V$ endowed with complex structures and the projection $G/W \to G/V$ may fail to be holomorphic \cite{GriffithsRoblesToledo2014}.

Regarding the non-compact homogeneous almost K\"ahler manifold $(G/V,\omega,J)$, a number of questions are naturally posed: Is $J$ integrable? What is the Hermitian scalar curvature of $J$? Denoting by $\rho$ the Chern-Ricci form of $J$, does it hold $\rho = \lambda \omega$ for some constant $\lambda$? If not, is it possible to move $v$ inside $\mathfrak g$ in order to satisfy that equation? Given $G$, is it possible to classify all such $v$'s?
Given a discrete torsion-free subgroup $\Gamma \subset G$, does the quotient symplectic manifold $(\Gamma\backslash G/V,\omega_\Gamma)$ admit any integrable compatible almost complex structure?

We shall give answers to these questions showing that the almost K\"ahler grometry of the orbit $(G/V,\omega,J)$ is governed by a certain element $\varphi$ which belongs to the dominant Weyl chamber $C$ of a suitably chosen positive root system $\Delta^+$ for $\mathfrak g$.
In simple words, denoted by $\ell$ the rank of $\mathfrak g$, all the geometry of $(G/V,\omega,J)$ is encoded in the euclidean geometry of a non-zero vector $\varphi$ belonging to a convex cone $C$ in the $\ell$-dimensional euclidean vector space, and a bunch of other vectors (corresponding to elements of $\Delta^+$).

In this perspective one can expect that giving a classification of all orbits $(G/V,\omega,J)$ satisfying $\rho=\lambda\omega$ is not completely hopeless, at least when $G$ is simple.
Unfortunately the intricate combinatoric of the problem prevented us from getting that classification. 
Whereas one can algorithmically list all such orbits for a given simple $G$, guessing the general pattern seems to us out of reach at the moment.
On the other hand, we are able to provide several infinite families for classical simple Lie groups and \emph{a complete classification for all exceptional ones}.
For the latter, see relevant tables in Appendix \ref{sec::appendix}.
For the former we have the following
\begin{thm}\label{thm::mainintro}
Let $G$ be a non-compact real simple Lie group of classical type.
There exist orbits $(G/V,\omega,J)$ satisfying $\rho=\lambda\omega$ if the Lie algebras of $G$ and $V$, and the constant $\lambda$ are as in the following table
\begin{center}
\begin{tabular}{cccc}
$\mathfrak g$ & $\mathfrak v$ & $\lambda$ \\ \hline
$\mathfrak{su}(p,q)$ & $\mathfrak{su}(p) \oplus \mathfrak{su}(q) \oplus \mathbf R$ & $-1$ & $p,q \geq 1$ \\
$\mathfrak{so}(2p,q)$ & $\mathfrak{su}(p) \oplus \mathfrak{so}(q) \oplus \mathbf R$ & $p-q-1$ & $p,q \geq 1$\\
$\mathfrak{so}^*(2\ell)$ & $\mathfrak{su}(\ell) \oplus \mathbf R$ & $-1$ & $\ell \geq 4$ \\
$\mathfrak{sp}(p,q)$ & $\mathfrak{su}(p) \oplus \mathfrak{sp}(q) \oplus \mathbf R$ & $p-2q+1$ & $p,q \geq 1$\\
$\mathfrak{sp}(\ell,\mathbf R)$ & $\mathfrak{su}(\ell) \oplus \mathbf R$ & $-1$ & $\ell \geq 3$\\
\end{tabular}
\end{center}
\end{thm}

We stress that the implication of the theorem above cannot be reversed.
This happens for two reasons.
Firstly for the obvious fact that rescaling the orbit, hence $\omega$, has the effect of rescaling the constant $\lambda$ accordingly.
Secondly, because a simple Lie group can admit several orbits, with different stabilizer, satisfying $\rho=\lambda \omega$ (see tables in the Appendix \ref{sec::appendix}).

Our results on this classification problem are complementary with those recently obtained by Alekseevsky and Podest\`a, who treat the case when $\varphi$ is in \emph{general position} inside $C$. In other words, they consider the case when $L \subset G$ is a maximal torus, and classify all non-compact real simple Lie groups which admit adjoint orbits $(G/L,\omega,J)$ satisfying $\rho=\lambda\omega$ \cite[Theorem 1.1]{AlekseevskyPodesta2018}.

We decided to not include explicit examples in order to contain the length of the paper.
On the other hand, one can check that Hermitian symmetric spaces such as the Siegel upper half-space $Sp(2\ell,\mathbf R)/U(\ell)$ or $SU(p,q)/S(U(p) \times U(q))$ are included in the class of spaces that we are going to discuss. 
Furthermore, non-integrable examples are constituted by period domains of weight two $SO(2p,q)/U(p) \times SO(q)$ \cite[Sections 4.2.1 and 4.2.2]{DellaVedova2017} and also by homogeneous spaces for exceptional Lie groups like $G_{2(2)}/U(2)$, $E_{7(7)}/U(7)$, $E_{8(8)}/U(8)$, among many others.

Finally we give a brief description of the sections of the paper.

In section \ref{sec::structureofg} we recall the definition of the Kirillov-Kostant-Souriau symplectic form $\omega$ on $G/V$ and we present a summary of results that we need on the structure of $\mathfrak g$.

In section \ref{sec::definitionJ} we give the definition of the canonical complex structure $J$ (see Definition \ref{defn::cancomplexstructure}) and we discuss its main features (see Proposition \ref{prop:pseudoholo} and its corollaries) including the compatibility with $\omega$ (Proposition \ref{Porp::compatibilityJ}).

In Section \ref{sec::spacialitycondition} we write the equation $\rho = \lambda \omega$ as an equation $\varphi'=\lambda\varphi$, where $\varphi'$ depends only on the vector $\varphi \in C$ mentioned above. Moreover we show that if $\psi \in C$ is likewise associated to another adjoint orbit and it satisfies $\psi' = \mu \psi$, then the signs of $\lambda$ and $\mu$ are the same (Proposition \ref{prop::uniquenesstype}). 
Moreover, the number of solutions of $\psi'= \mu \psi$ is at most one, a finite number or infinite in accordance with the sign of $\lambda$ being negative, positive or zero respectively (Proposition \ref{prop::finitenessuptoscaling}).
Finally we give a precise characterization of $\varphi$ when it is a solution of $\varphi'=\lambda\varphi$ (Theorems \ref{thm::specialvFGT} and \ref{thm::specialvCY}).

In Section \ref{sec::Hermitianscalarcurvature} we provide a formula for the Hermitian scalar curvature of $(G/V,\omega,J)$, which is necessarily constant due to homogeneity, in terms of $\varphi$ and root data (Lemma \ref{lem::Hermitianscalarcurvature}).

In Section \ref{sec::Nijenhuis} we discuss the integrability of $J$ and we express its Nijenhuis tensor in terms of $\varphi$ and root data (Lemma \ref{lem::NH} and Theorem \ref{thm::|N|^2}).

In section \ref{sec::compactquotients} we consider uniform lattices $\Gamma \subset G$ and compact quotients $(\Gamma \backslash G / V, \omega_\Gamma, J_\Gamma)$. Since $J_\Gamma$ is not integrable in general, we ask when the compact symplectic manifold $(\Gamma \backslash G / V, \omega_\Gamma)$ is of K\"ahler type, i.e. it admits a compatible complex structure. We do not give a full answer, but we show that it is not of K\"ahler type whenever $\Gamma \backslash G / V$ is symplectic Fano, meaning that its first Chern class $c_1$ is a positive multiple of $[\omega_\Gamma]$ (Lemma \ref{lem: sympFanoNotKahler}).
We also briefly discuss our expectation that $(\Gamma \backslash G / V, \omega_\Gamma)$ is of K\"ahler type precisely when $J$ is integrable. 

In section \ref{sec::Vogandiagrams} we use Vogan diagrams as a tool for making some steps toward the classification of all possible adjoint orbits $(G/V,\omega,J)$ satisfying $\rho = \lambda \omega$. 
In particular we discuss how to look at these diagrams in order to find solutions of the equation $\varphi'=\lambda\omega$ and to show that $J$ is not integrable for most of them (Proposition \ref{prop::singlenodespecial} and Theorems \ref{thm::integrabilityonenode} and \ref {thm::mostlynonintegrable}).

Finally, in Appendix \ref{sec::appendix}, we list all solutions of the equation $\varphi'=\lambda\varphi$ associated to connected Vogan diagrams either of rank at most $\ell=4$ or of exceptional type.

\bigskip

The results of this paper are part of the research project of A.G. under the direction of A.D.V within the `Joint PhD Program in Mathematics Milano Bicocca - Pavia - INdAM'.
A.G. has given talks on the results of this paper within the conferences `Hamiltonian PDEs: Models and Applications', Milano - Bicocca, June 2018, and `Progressi Recenti in Geometria Reale e Complessa - XI', CIRM, September 2018.

A very recent preprint of Alekseevsky and Podest\`a \cite{AlekseevskyPodesta2018} partially overlaps with some results of this work. The authors of this work became aware of that work when it appeared on the arXiv.

\section{Adjoint orbits and structure of $\mathfrak g$}\label{sec::structureofg}

The aim of this section is to recall some elementary facts about adjoint orbits of semisimple Lie groups together with some structure theory of real semisimple Lie algebras, mainly to fix the notation.
Most of facts presented here are extensively discussed in \cite[Section III]{Helgason1978} and \cite[Sections II and VI]{Knapp1996}. 
Our exposition follows quite closely the one of Griffiths and Schmid \cite[Section 3]{GriffithsSchmid1969}.

Let $G$ be a non-compact real semisimple Lie group with Lie algebra $\mathfrak g$.
Consider the adjoint action of $G$ on $\mathfrak g$, and let be chosen $v \in \mathfrak g$ having compact isotropy subgroup $V \subset G$.
The Lie algebra of $V$ turns out to be equal to
\begin{equation}
\mathfrak v = \{x \in \mathfrak g \mbox{ s.t. } [v,x]=0 \}.
\end{equation}
Due to compactness of $V$, the Killing form
\begin{equation}
B(x,y) = \tr(\ad(x) \ad(y)) \quad x,y \in \mathfrak g
\end{equation}
restricts to a negative definite scalar product on $\mathfrak v$.
On the other hand, the orthogonal complement
\begin{equation}
\mathfrak m = \{x \in \mathfrak g \mbox{ s.t. } B(x,y)=0 \mbox{ for all } y \in \mathfrak v\}
\end{equation}
is canonically isomorphic to the tangent space at the identity coset $e$ of the adjoint orbit $G/V$ of $v$.

By the Cartan criterion of semisimplicity, the Killing form is non-degenerate.
Therefore it induces a canonical isomorphism between $\mathfrak g$ and its dual $\mathfrak g^*$.
In particular $v \in \mathfrak g$ corresponds to $\nu \in \mathfrak g^*$, defined by $\nu(x)=B(v,x)$ for all $x \in \mathfrak g$. 
Due to $G$-invariance of $B$, this canonical isomorphism is $G$-equivariant.
As a consequence, the co-adjoint orbit of $\nu$ is $G$-equivariantly diffeomorphic to the adjoint orbit $G/V$ of $v$.
Moreover, $G/V$ turns out to be equipped with the Kirillov-Kostant-Souriau symplectic form $\omega$, which is homogeneous and, at the identity coset $e$, it corresponds  with the symplectic form $\sigma$ on $\mathfrak m$ defined by
\begin{equation}
\sigma(x,y) = B(v,[x,y]) \quad x,y \in \mathfrak m.
\end{equation}

More details on the relationship between the symplectic form $\omega$ on $G/K$ and $\sigma$ are discussed in several places (e.g. \cite[Section 3]{DellaVedova2017}).

In the rest of this section we discuss the structure of $\mathfrak g$.
To this end, let $\mathfrak g_c$ be the complex Lie algebra obtained by complexifying $\mathfrak g$. 
Clearly $\mathfrak g \subset \mathfrak g_c$ is fixed by a unique complex conjugation $\tau$ on $\mathfrak g_c$.
Let $\mathfrak k \subset \mathfrak g$ be a maximal compact subalgebra such that $\mathfrak v \subset \mathfrak k$.
A Cartan decomposition
\begin{equation}
\mathfrak g = \mathfrak k \oplus \mathfrak p,
\end{equation}
can be produced by considering the complexification $\mathfrak k_c \subset \mathfrak g_c$ of $\mathfrak k$ and its $\ad(\mathfrak k_c)$-invariant complement $\mathfrak p_c$, and letting $\mathfrak p = \mathfrak p_c \cap \mathfrak g$. 
With this at hand, a compact real form of $\mathfrak g_c$ is given by 
\begin{equation}
\mathfrak g_0 = \mathfrak k \oplus i \mathfrak p.
\end{equation}

Denote by $\tau_0$ the complex conjugation induced on $\mathfrak g_c$ by $\mathfrak g_0$, and let $\mathfrak h_0$ be a maximal abelian subalgebra of $\mathfrak k$ such that $v \in \mathfrak h_0$.
By the hypothesis that $v$ has compact stabilizer, one can deduce that $\mathfrak k$ and $\mathfrak g_0$ have the same rank, hence complexifying $\mathfrak h_0$ yields a $\tau_0$-invariant Cartan subalgebra $\mathfrak h_c$ of $\mathfrak g_c$.

The adjoint representation of $\mathfrak h_c$ on $\mathfrak g_c$ gives a decomposition
\begin{equation}
\mathfrak g_c = \mathfrak h_c \oplus \sum_{\alpha \in \Delta} \mathfrak g^\alpha,
\end{equation}
where the set of roots $\Delta$ is a finite subset of the dual space of $\mathfrak h_c$, and each root space
\begin{equation}\label{eq::root-space}
\mathfrak g^\alpha = \left\{ x \in \mathfrak g_c \mbox{ s.t. } [h,x] = \alpha(h)x \mbox{ for all } h \in \mathfrak h_c \right\}, 
\end{equation}
is one-dimensional.
Moreover, if $\alpha,\beta \in \Delta$, then one has $[\mathfrak g^\alpha,\mathfrak g^\beta] = \mathfrak g^{\alpha+\beta}$ whenever $\alpha+\beta \in \Delta$, and $[\mathfrak g^\alpha,\mathfrak g^\beta] = 0$ otherwise.

Due to the fact that $\mathfrak g_0$ is a compact real form of $\mathfrak g_c$, all roots assume real values on the real vector space $\mathfrak h_{\mathbf R} = i\mathfrak h_0$.
Therefore one can regard $\Delta$ as a subset of the dual space $\mathfrak h_{\mathbf R}^*$.
Since $\mathfrak h_{\mathbf R}$ is purely imaginary with respect to $\tau_0$, one has
\begin{equation}
\tau_0(\mathfrak g^\alpha) = \mathfrak g^{-\alpha}.
\end{equation}

The complex conjugations $\tau$ and $\tau_0$, induced on $\mathfrak g_c$ respectively by $\mathfrak g$ and $\mathfrak g_0$, commute.
Moreover, the composition $\theta = \tau \tau_0$ is an involutive automorphism of $\mathfrak g_c$ whose $1$ and $-1$ eigenspaces are $\mathfrak k_c$ and $\mathfrak p_c$ respectively. 
Since $\mathfrak h_c$ is contained in $\mathfrak k_c$, the adjoint action of the Cartan subalgebra $\mathfrak h_c$ commutes with $\theta$.
This implies that any root space $\mathfrak g^\alpha$ is contained either in $\mathfrak k_c$ or in $\mathfrak p_c$, and the root $\alpha$ is called compact or non-compact accordingly.
Set $\varepsilon_\alpha = -1$ if $\alpha$ is compact and $\varepsilon_\alpha = 1$ otherwise.
One can check that $\varepsilon_\alpha = \varepsilon_{-\alpha}$ for all $\alpha \in \Delta$, and that $\varepsilon_{\alpha+\beta} = - \varepsilon_\alpha \varepsilon_\beta$ whenever $\alpha, \beta, \alpha+\beta \in \Delta$.

A positive root system is a subset $\Delta^+ \subset \Delta$ such that 
\begin{enumerate}\renewcommand{\labelenumi}{\alph{enumi})}
\item for all $\alpha \in \Delta$, either $\alpha$ or $-\alpha$ belongs to $\Delta^+$,
\item if $\alpha, \beta \in \Delta^+$ and $\alpha+\beta \in \Delta$, then $\alpha + \beta \in \Delta^+$.
\end{enumerate}
A positive root is called simple if it cannot be written as a sum $\alpha + \beta$ where $\alpha, \beta \in \Delta^+$. 
Once a positive root system $\Delta^+$ is fixed, the set of simple roots $\Sigma^+ \subset \Delta^+$ turns out to be a basis for $\mathfrak h_{\mathbf R}^*$. Moreover, if $\alpha$ is a root, then $\alpha = \sum_{\gamma \in \Sigma^+} n_\gamma \gamma$, where $n_{\gamma} \in \mathbf Z$ are either all positive or all negative according to the fact that $\alpha$ is a positive or a negative root.
The compactness of such a root is determined by the compactness of $\gamma \in \Sigma^+$ and the integers $n_\gamma$'s.
More precisely, by induction on the formula recalled above for the compactness of a sum of two roots, one has the following
\begin{lem}\label{lem::compactnessroots}
If a root has the form $\alpha = \sum_{\gamma \in \Sigma^+} n_\gamma \gamma$ for suitable $n_\gamma \in \mathbf Z$,
then 
\begin{equation}
\varepsilon_\alpha = (-1)^{1+\sum_{\gamma \in \Sigma^+}n_\gamma} \prod_{\gamma \in \Sigma^+} \varepsilon_\gamma^{n_\gamma}.
\end{equation}
\end{lem}
Note that compactness of simple roots induces a splitting
\begin{equation*}
\Sigma^+ = \Sigma^+_c \cup \Sigma^+_n
\end{equation*}
where $\Sigma^+_c = \{ \gamma \in \Sigma^+ \, | \, \varepsilon_\gamma = -1 \}$ is the set of compact simple roots, and $\Sigma^+_n = \{ \gamma \in \Sigma^+ \, | \, \varepsilon_\gamma = 1 \}$ is the set of non-compact ones.
Therefore, lemma above reduces to the following
\begin{cor}\label{cor::epsilonalphaformula}
If a root has the form $\alpha = \sum_{\gamma \in \Sigma^+} n_\gamma \gamma$, then 
\begin{equation}
\varepsilon_\alpha = (-1)^{1+\sum_{\gamma \in \Sigma^+_n}n_\gamma}.
\end{equation}
\end{cor}

Note that the Killing form $B$ restricts to a positive scalar product on $\mathfrak h_{\mathbf R}$. 
As a consequence one has an isomorphism between $\mathfrak h_{\mathbf R}^*$ and $\mathfrak h_{\mathbf R}$ which takes $\psi$ to the unique $h_\psi$ such that $\psi(h) = B(h_\psi,h)$ for all $h \in \mathfrak h_{\mathbf R}$.
Therefore, is defined a positive scalar product on $\mathfrak h_{\mathbf R}^*$ by letting 
\begin{equation*}
(\psi,\psi') = B(h_{\psi},h_{\psi'}).
\end{equation*}
The set of hyperplanes $P_\alpha = \{ \psi \in \mathfrak h_{\mathbf R}^* \,|\, (\psi,\alpha)=0\}$, with $\alpha \in \Delta$, divides $\mathfrak h_{\mathbf R}^*$ into a finite number of closed convex cones, named Weyl chambers.
To each positive root system $\Delta^+$ it corresponds a dominant Weyl chamber defined by
\begin{equation}
C = \left\{ \psi \in \mathfrak h_{\mathbf R}^* \,|\, (\psi,\alpha) \geq 0 \mbox{ for all } \alpha \in \Delta^+ \right\},
\end{equation}
and this correspondence is bijective.

Recall that we have chosen $\mathfrak h_0$ so that $v$ belongs to it, or equivalently, so that $iv$ belongs to $\mathfrak h_{\mathbf R}$. 
Therefore is uniquely defined a co-vector
\begin{equation}
\varphi \in \mathfrak h_{\mathbf R}^* \mbox{ such that } h_\varphi = -iv.
\end{equation}
Moreover, note that one can choose a positive root system $\Delta^+$ such that $(\varphi,\alpha) \geq 0$ for all $\alpha \in \Delta^+$ or, equivalently, such that $\varphi$ belongs to the dominant Weyl chamber $C$ associated to $\Delta^+$. 
We assume henceforth that we make such a choice of a positive root system $\Delta^+$.

A convenient way for parametrizing all elements of $C$ is by means of fundamental dominant weights, which we now recall.
Let us denote by $\ell$ the rank of $\mathfrak g$, that is the dimension of $\mathfrak h_0$.
Therefore we can label simple roots form 1 to $\ell$ so that $\Sigma^+ = \{\gamma_1,\dots,\gamma_\ell\}$.
Moreover, let $A = (A_{ij})$ be the associated Cartan matrix, whose coefficients are given by
\begin{equation*}
A_{ij} = \frac{2 (\gamma_i,\gamma_j)}{(\gamma_i,\gamma_i)}.
\end{equation*}
The fundamental dominant weights $\varphi_1,\dots, \varphi_\ell$ are the linear forms on $\mathfrak h_{\mathbf R}$ defined by $\varphi_j = \sum_{i=1}^\ell A^{ij}\gamma_i$, where $(A^{ij}) = A^{-1}$.
Clearly they form a basis of $\mathfrak h_{\mathbf R}^*$.
As shown by the following lemma, this basis turns out to be very well behaved for our purposes.  

\begin{lem}\label{lem::vassumoffunddomweights}
Let $\psi \in \mathfrak h_{\mathbf R}^*$ and write $\psi = \sum_{j=1}^\ell w^j \varphi_j$ for some reals $w^1, \dots, w^\ell$.
Then one has $(\psi,\alpha) \geq 0$ for each positive root $\alpha$ if and only if all $w^i$'s are non-negative.
Moreover one has $\Delta^+ \setminus \psi^\perp = \Span \left\{ \gamma_j | w^j \neq 0\right\} \cap \Delta^+$.
\end{lem}
\begin{proof}
For all $\alpha \in \Delta^+$ of the form $\alpha = \sum_{i=1}^\ell n^i \gamma_i$ for suitable non-negative integers $n^1, \dots, n^\ell$, one readily calculates 
\begin{equation}\label{eq::psiscalalpha}
(\psi,\alpha) = \frac{1}{2} \sum_{i=1}^\ell w^i n^i |\gamma_i|^2,
\end{equation}
whence the thesis follows.
\end{proof}

As a consequence of the above lemma, the dominant Weyl chamber $C$ turns out to be the (closed) convex cone spanned by the fundamental dominant weights $\varphi_1,\dots,\varphi_\ell$.
In particular, one can write
\begin{equation}
\varphi = \sum_{i=1}^\ell v^i \varphi_i \quad \mbox{ for some } \quad v^1,\dots,v^\ell \geq 0.
\end{equation}

We end this section by recalling the following (see e.g. \cite[Theorem 6.6]{Knapp1996} and \cite[Section 3]{GriffithsSchmid1969})
\begin{thm}\label{thm::basisgc}
For each root $\alpha \in \Delta$, one can choose $e_\alpha \in \mathfrak g^{\alpha}$ such that
\begin{enumerate}
\item $[e_\alpha,e_{-\alpha}] = h_\alpha$,
\item $B(e_\alpha, e_\beta) = \delta_{\alpha,-\beta}$, \label{item::Bab}
\item $B(h_\alpha,h) = \alpha(h)$ for all $h \in \mathfrak h_c$, \label{item::Bhalphah}
\item $[e_\alpha,e_\beta] = 0$ if $\alpha+\beta \neq 0$ and $\alpha+\beta \notin \Delta$,
\item $[e_\alpha,e_\beta] = N_{\alpha,\beta}e_{\alpha+\beta}$ if $\alpha+\beta \in \Delta$,
\item $\tau_0(e_\alpha) = -e_{-\alpha}$,
\item $\tau(e_\alpha) = \varepsilon_\alpha e_{-\alpha}$, \label{item::tauealpha}
\end{enumerate} 
where $N_{\alpha,\beta} \in \mathbf R$ are non-zero and satisfy $N_{\beta,\alpha} = - N_{\alpha,\beta} = N_{-\alpha,-\beta} = N_{-\beta,\alpha+\beta} = N_{\alpha+\beta,-\alpha}$.
 \end{thm}
Note that one can conveniently define $N_{\alpha,\beta}=0$ whenever $\alpha+\beta \neq 0$ and it is not a root, so that by abuse of notation one can write $[e_\alpha,e_\beta] = N_{\alpha,\beta}e_{\alpha+\beta}$ for all $\alpha,\beta \in \Delta$ even if $\alpha+\beta$ is not a root (hence $e_{\alpha+\beta}$ is undefined). This will be useful in some computations below.

\section{Canonical almost complex structure}\label{sec::definitionJ}

In this section we recall the definition of a canonical homogeneous almost complex structure on $G/V$ which turns out to be compatible with the Kirillov-Kostant-Souriau symplectic form $\omega$ \cite{DellaVedova2017,AlekseevskyPodesta2018}.

For each root $\alpha \in \Delta$ consider the real number 
\begin{equation}\label{eq::deflambdaalpha}
\lambda_\alpha = s_\alpha (\alpha,\varphi)
\end{equation} where $s_\alpha$ is the signum of $\alpha$, meaning that $s_\alpha = +1$ if $\alpha$ is a positive root, and $s_\alpha = -1$ otherwise.
Note that by our assumption on the Weyl chamber one has $\lambda_\alpha \geq 0$, and by definition above one has $\lambda_\alpha=0$ precisely when $\alpha$ is orthogonal to $\varphi$.
Moreover define
\begin{equation}\label{eq::uandvintermsofe}
u_\alpha = \frac{i^{(1-\varepsilon_\alpha)/2}}{\sqrt 2} (e_\alpha + e_{-\alpha}), \qquad
v_\alpha = \frac{i^{(3-\varepsilon_\alpha)/2}}{\sqrt 2} s_\alpha (e_\alpha - e_{-\alpha}).
\end{equation}
Clearly one has $u_\alpha = u_{-\alpha}$ and similarly $v_\alpha = v_{-\alpha}$.
Moreover one has the following
\begin{lem}\label{lem::advonualphaandvalpha}
For all roots $\alpha, \beta \in \Delta$ one has
\begin{enumerate}
\item $B(u_\alpha,u_\beta) = B(v_\alpha,v_\beta) = (\delta_{\alpha,\beta} + \delta_{\alpha,-\beta})\varepsilon_\alpha$, \label{item::Bualphaubeta}
\item $B(u_\alpha,v_\beta)=0$, \label{item::Bualphavbeta}
\item $u_\alpha, v_\alpha \in \mathfrak g$, \label{item::alphavalphaing}
\item $[v,u_\alpha] = \lambda_\alpha v_\alpha$ and $[v,v_\alpha] = -\lambda_\alpha u_\alpha$. \label{item::[v,ualpha]}
\end{enumerate}
\end{lem}
\begin{proof}
These are simple consequences of Theorem \ref{thm::basisgc}.
In particular, items \ref{item::Bualphaubeta} and \ref{item::Bualphavbeta} follow by item \ref{item::Bab} of the theorem after some easy calculations.

On the other hand, in order to prove item \ref{item::alphavalphaing} one has just to show that $u_\alpha$ and $v_\alpha$ are invariant with respect to the complex conjugation $\tau$.
This follows easily by item \ref{item::tauealpha} of Theorem \ref{thm::basisgc}.

Finally, item \ref{item::[v,ualpha]} above follows by definition of root-space \eqref{eq::root-space}, by item \ref{item::Bhalphah} of Theorem \ref{thm::basisgc}, and by the fact that $h_\varphi = -iv$.
\end{proof}

A straightforward consequence of lemma above is the $B$-orthogonal decomposition
\begin{equation}
\mathfrak g = \mathfrak h_0 \oplus \sum_{\alpha \in \Delta^+} \Span\{u_\alpha ,v_\alpha\}.
\end{equation}
Note that it can be refined by splitting the set of positive roots $\Delta^+$ into the subset of positive roots $\Delta^+ \cap \varphi^\perp$ which are orthogonal to $\varphi$, and the set $\Delta^+ \setminus \varphi^\perp$ of those which are not orthogonal to $\varphi$.
By \eqref{eq::deflambdaalpha}, a root $\alpha$ belongs either to $\Delta^+ \cap \varphi^\perp$ or $\Delta^+ \setminus \varphi^\perp$ according to $\lambda_\alpha=0$ or $\lambda_\alpha>0$.
Therefore, by item \ref{item::[v,ualpha]} of Lemma \ref{lem::advonualphaandvalpha} it follows that the Lie algebra $\mathfrak v$ of the stabilizer of $v$ decomposes as the direct sum of $\mathfrak h_0$ with the subspace of $\mathfrak g$ spanned by $u_\alpha, v_\alpha$ as $\alpha$ runs in $\Delta^+ \cap \varphi^\perp$.
Since $V$ is compact, all roots that belong to $\Delta^+ \cap \varphi^\perp$ must be compact. 
On the other hand, the subspace of $\mathfrak g$ spanned by $u_\alpha, v_\alpha$ for all $\alpha \in \Delta^+ \setminus \varphi^\perp$ turns out to be $\mathfrak m$.
From this it follows readily a formula for the dimension of the adjoint orbit $G/V$ of $v$ in terms of $\varphi$ and the positive roots.
Summarizing one has the following
\begin{prop}\label{prop::decompvplusm}
The summands of the $B$-orthogonal decomposition $\mathfrak g = \mathfrak v \oplus \mathfrak m$ are given by
\begin{equation}\label{eq::kandmdecomposition}
\mathfrak v = \mathfrak h_0 \oplus \sum_{\alpha \in \Delta^+ \cap \varphi^\perp} \Span\{u_\alpha, v_\alpha\}, \qquad
\mathfrak m = \sum_{\alpha \in \Delta^+ \setminus \varphi^\perp} \Span\{u_\alpha, v_\alpha\}.
\end{equation}
Moreover, all roots of $\Delta^+ \cap \varphi^\perp$ are compact and one has the dimension formula
\begin{equation}\label{eq::dimensionorbit}
\dim G/V = \dim \mathfrak g - \ell - 2 \card\left\{ \alpha \in \Delta^+ | (\alpha,\varphi)=0 \right\},
\end{equation}
where $\ell = \dim \mathfrak h_0$ is the rank of $\mathfrak g$.
\end{prop}

As anticipated at the beginning of this section, here we define a canonical homogeneous almost complex structure $J$ on $G/V$.
In order to do this let $H$ be the complex structure on $\mathfrak h_0^\perp = \Span_{\alpha \in \Delta^+} \{u_\alpha,v_\alpha\}$ defined by
\begin{equation}\label{eq::defH}
Hu_\alpha = \varepsilon_\alpha v_\alpha, \qquad Hv_\alpha = - \varepsilon_\alpha u_\alpha \qquad \mbox{for all } \alpha \in \Delta.
\end{equation}
Clearly this complex structure makes complex the splitting 
\begin{equation*}
\mathfrak h_0^\perp = \Span_{\alpha \in \Delta^+ \cap \varphi^\perp} \{u_\alpha,v_\alpha\} \oplus \Span_{\alpha \in \Delta^+ \setminus \varphi^\perp} \{u_\alpha,v_\alpha\}.
\end{equation*}
Note that the second summand is $\mathfrak m$.
\begin{defn}\label{defn::cancomplexstructure}
The canonical almost complex structure $J$ on the orbit $G/V$ is the homogeneous structure induced by $H$ on $\mathfrak m$.
\end{defn}

\begin{rem}\label{rem::equivtoADV}
If one decomposes $x \in \mathfrak m$ as $x = \sum_{\alpha \in \Delta^+ \setminus \varphi^\perp} x_\alpha$, where $x_\alpha$ is the component of $x$ along the real two-dimensional space spanned by $u_\alpha,v_\alpha$, then by item \ref{item::[v,ualpha]} of Lemma \ref{lem::advonualphaandvalpha} it follows that
\begin{equation}\label{eq::Hxdecomp}
H x = \sum_{\alpha \in \Delta^+ \setminus \varphi^\perp} \frac{\varepsilon_\alpha}{\lambda_\alpha}[v,x_\alpha] \qquad x \in \mathfrak m.
\end{equation}
Therefore $J$ on $G/V$ coincides with the almost complex structure studied by the first author \cite[Subsection 4.2]{DellaVedova2017}.

On the other hand, the almost complex structure induced by $H$ has been recently studied also by Alekseevsky and Podest\`a \cite{AlekseevskyPodesta2018}.
In particular, when $G$ is simple, they show that there are no other almost complex structures on $G/V$ which are both compatible with $\omega$ and homogeneous.
\end{rem}
The complex structure defined above is compatible with $G$-equivariant projections of adjoint orbits.
More precisely one has the following
\begin{prop}\label{prop:pseudoholo}
Let $\tilde \varphi$ be an element of $\mathfrak h_{\mathbf R}^*$ such that for all roots $\alpha \in \Delta$ one has $(\tilde \varphi,\alpha)=0$ whenever $(\varphi,\alpha)=0$.
Therefore the stabilizer $\tilde V$ of $\tilde v = - i h_{\tilde \varphi}$ contains $V$ and the induced projection $\pi : G/V \to G/ \tilde V$ is pseudo-holomorphic with respect to the almost complex structures $\tilde J$ and $J$ induced by $H$ on the adjoint orbits of $\tilde v$ and $v$ respectively.
\end{prop}
\begin{proof}
By Proposition \ref{prop::decompvplusm}, the Lie algebra of the stabilizer of $\tilde v$ is given by
\begin{equation}
\tilde {\mathfrak v} = \mathfrak h_0 \oplus \sum_{\alpha \in \Delta^+ \cap \tilde \varphi^\perp} \Span\{u_\alpha, v_\alpha\},
\end{equation}
which contains $\mathfrak v$ in force of the hypothesis on $\tilde \varphi$.
This proves that $V \subset \tilde V$.
The induced projection $\pi: G/V \to G/\tilde V$ is clearly $G$-equivariant, thus in order to prove that it is pseudo-holomorphic it is enough to prove that its differential $d_e\pi$ at the identity coset $e \in G/V$ is complex-linear.
To this end, note that
\begin{equation}
\tilde {\mathfrak m} = \mathfrak h_0 \oplus \sum_{\alpha \in \Delta^+ \setminus \tilde \varphi^\perp} \Span\{u_\alpha, v_\alpha\}
\end{equation}
is contained in $\mathfrak m$ and that the differential $d_e\pi$ corresponds to the (B-orthogonal) projection from $\mathfrak m$ to $\tilde{\mathfrak m}$, which by \eqref{eq::defH} intertwines the complex structures induced on $\mathfrak m$ and $\tilde{\mathfrak m}$ by $H$.
\end{proof}
In the situation of the proposition above, choosing $\tilde \varphi \in C$ being orthogonal to no roots yields the following
\begin{cor}\label{cor::G/TG/VG/Kbis}
Associated with $v$ there exists $w \in \mathfrak g$ whose stabilizer is a maximal torus $T \subset G$ such that $T \subset V$.
Once the orbit $G/T$ is equipped with the invariant almost complex structure induced by $H$, the natural projection
\begin{equation*}
G/T \to G/V
\end{equation*}
is pseudo-holomorphic.
\end{cor}

Note that in general, $w$ as in the statement above is far from being unique, in that it corresponds to choosing $\tilde \varphi$ in an open dense subset of $C$.
On the other hand, sometimes one can choose $\tilde \varphi \in C$ in a very special position so that Proposition above yields

\begin{cor}\label{cor::G/TG/VG/K}
If there is a nonzero $\tilde \varphi \in C$ which is orthogonal to all compact roots, then there exists $v_0 \in \mathfrak g$ whose stabilizer is a maximal compact $K \subset G$ such that $V \subset K$.
Once the orbit $G/K$ is equipped with the invariant almost complex structure induced by $H$, the natural projection
\begin{equation*}
G/V \to G/K
\end{equation*}
is pseudo-holomorphic.
\end{cor}

We end this section by noting that the canonical almost complex structure $J$ on $G/V$ is compatibile with the Kirillov-Kostant-Souriau symplectic form $\omega$. This follows by Remark \ref{rem::equivtoADV} and \cite[Subsection 4.2]{DellaVedova2017}. For convenience of the reader, we repeat here the argument.

\begin{prop}\label{Porp::compatibilityJ}
The canonical almost complex structure $J$ on $G/V$ is compatibile with the Kirillov-Kostant-Souriau symplectic form $\omega$.
\end{prop}
\begin{proof}
Since both $\omega$ and $J$ are homogeneous, it is enough to check their compatibility at the identity coset $e \in G/V$. Thus we are reduced to prove the compatibility of the complex structure $H$ on $\mathfrak m$ with the linear symplectic form $\sigma$.
For all $x \in \mathfrak m$, denote by $x_\alpha$ the component of $x$ along the real subspace spanned by $u_\alpha,v_\alpha$ with $\alpha \in \Delta^+ \setminus \varphi^\perp$, so that $Hx$ takes the form \eqref{eq::Hxdecomp}.
Therefore, by definition of $\sigma$ and $G$-invariance of the Killing form $B$, for all $x, y \in \mathfrak m$ one has
\begin{equation*}
\sigma(x,Hy) = \sum_{\alpha,\beta \in \Delta^+ \setminus \varphi^\perp} \frac{\varepsilon_\beta}{\lambda_\beta}B([v,x_\alpha],[v,y_\beta]).
\end{equation*}
Now observe that by items \ref{item::Bualphaubeta} and \ref{item::Bualphavbeta} of Lemma \ref{lem::advonualphaandvalpha} it follows that
\begin{equation}
\sigma(x,Hy) = \sum_{\alpha \in \Delta^+ \setminus \varphi^\perp} \frac{\varepsilon_\alpha}{\lambda_\alpha}B([v,x_\alpha],[v,y_\alpha]),
\end{equation}
which is clearly symmetric in $x$ and $y$, whence $\sigma(x,Hy) + \sigma(Hx,y)=0$.
Moreover, taking $y=x$ and writing $x_\alpha = x_\alpha^u u_\alpha + x_\alpha^v v_\alpha$, by items \ref{item::Bualphaubeta} and \ref{item::[v,ualpha]} of Lemma \ref{lem::advonualphaandvalpha} it follows
\begin{equation}
\sigma(x,Hx) = \sum_{\alpha \in \Delta^+ \setminus \varphi^\perp} \lambda_\alpha \left( (x_\alpha^u)^2 + (x_\alpha^v)^2\right),
\end{equation}
which is positive as soon as $x$ is non-zero.
\end{proof}

As a consequence of the proof and Lemma \ref{lem::advonualphaandvalpha} one has the following

\begin{cor}\label{cor::symplecticorthonormalbasis}
The vectors $(1/\sqrt{\lambda_\alpha}) u_\alpha$, $(\varepsilon_\alpha/\sqrt{\lambda_\alpha}) v_\alpha$ with $\alpha \in \Delta^+ \setminus \varphi^\perp$ constitute a symplectic basis of $\mathfrak m$ which is also orthonormal with respect to the scalar product induced by $H$ and $\sigma$.
\end{cor}

\section{The condition $\rho = \lambda \omega$}\label{sec::spacialitycondition}

Let $(M,\omega)$ be a symplectic manifold, and let $J$ be a compatible almost complex structure on it.  
The Chern-Ricci form of $J$ is a closed two-form $\rho$ on $M$ defined as follows. 
Consider the Chern connection of $J$, that is the unique affine connection $\nabla$ on $M$ which satisfies $\nabla \omega = 0$, $\nabla J = 0$ and whose torsion coincides with the Nijenhuis tensor $N$ of $J$ (see Section \ref{sec::Nijenhuis} for the definition of $N$).
The curvature $R$ of $\nabla$ is then a two-form with values in $\End(TM)$, and $\rho$ is defined by the identity
\begin{equation*}
\rho(X,Y) = \tr(JR(X,Y)).
\end{equation*}

A this point one can ask if the equation $\rho = \lambda \omega$ is satisfied for some constant $\lambda$.
If this is the case and in addition $J$ is integrable, then $(M,\omega,J)$ is K\"ahler-Einstein.
In particular, the Riemannian metric $g$ associated to $\omega$ and $J$ is both K\"ahler and Einstein.
On the other hand, the terminology when $J$ is not assumed to be integrable is less standard. 
Indeed, a general almost K\"ahler manifold $(M,\omega,J)$ which satisfies $\rho = \lambda \omega$ is sometimes called \emph{Hermitian-Einstein}, or \emph{special} \cite{DellaVedova2017}, or \emph{Chern-Einstein} \cite{AlekseevskyPodesta2018}.
Note that in this case, the metric $g$ associated to $\omega$ and $J$ is rarely Einstein.

In this section we consider an adjoint orbit $(G/V,\omega,J)$ equipped with its canonical almost-K\"ahler structure (see Section \ref{sec::structureofg} and Definition \ref{defn::cancomplexstructure}) and we characterize when it satisfies the condition $\rho = \lambda \omega$.

As shown in \cite[Subsection 4.2]{DellaVedova2017}, in the same way as the symplectic form $\omega$ is determined by $\sigma$ and $v$, the Chern-Ricci form $\rho$ of $J$ is induced by the two form on $\mathfrak m$ defined by $B(v',[x,y])$, where $v'= 2 \sum_{\alpha \in \Delta^+ \setminus \varphi^\perp} [u_\alpha,v_\alpha]$.
By definition \eqref{eq::uandvintermsofe} of $u_\alpha,v_\alpha$, the latter turns out to be equal to
\begin{equation}\label{eq::v'sumofroots}
v' 
= -2i \sum_{\alpha \in \Delta^+ \setminus \varphi^\perp} \varepsilon_\alpha h_\alpha.
\end{equation}
As a consequence, letting
\begin{equation}\label{eq::definitionvarphi'}
\varphi' = - 2 \sum_{\alpha \in \Delta^+ \setminus \varphi^\perp} \varepsilon_\alpha \alpha
\end{equation}
defines an element of the root lattice $\varphi' \in \mathfrak h_{\mathbf R}^*$ such that $h_{\varphi'} = -iv'$.
The condition on the Chern-Ricci form $\rho = \lambda \omega$ is then equivalent to $v' = \lambda v$ and also to $\varphi' = \lambda \varphi$.
Note that $\varphi'$ depends \emph{discretely} on $\varphi$ in the sense that it is only sensitive to which (necessarily compact) roots $\varphi$ is orthogonal to, but not to the distance of $\varphi$ from those roots.

This equation can be made more tractable from the combinatorial point of view by introducing the element of the root lattice $\eta \in \mathfrak h_{\mathbf R}^*$ defined by
\begin{equation*}
\eta = -2 \sum_{\alpha \in \Delta^+} \varepsilon_\alpha \alpha.
\end{equation*}
Note that it depends only on the semisimple Lie algebra $\mathfrak g$ and on the chosen set of positive roots.
Clearly one can write the difference $\eta - \varphi'$ as a linear combination of roots.
Therefore, recalling that no non-compact roots can be orthogonal to $\varphi$, the condition $\rho = \lambda \omega$ takes the form
\begin{equation}\label{eq::rho=lambdaomegaintermsoftau}
\eta - 2 \sum_{\alpha \in \Delta^+ \cap \varphi^\perp} \alpha = \lambda \varphi.
\end{equation}
Since $\varphi$ cannot be zero, and the sum is performed over all roots $\alpha$ which are orthogonal to $\varphi$, one readily gets the following
\begin{lem}\label{lem::lambdaintermsofvarphiandtheta}
If $\varphi'= \lambda \varphi$, then $\lambda = (\eta,\varphi) / |\varphi|^2$.
\end{lem}

Up to now we considered $v$, hence $\varphi$, as given and fixed.
On the other hand, in this section we are primarily interested in solving the equation $\rho = \lambda \omega$, that is in finding $\varphi$ such that $\varphi' = \lambda \varphi$.
In this perspective, the next result guarantees that the sign of the constant $\lambda$ is uniquely determined by the semisimple Lie algebra $\mathfrak g$ (hence by the group $G$), and by the choice of the dominant Weyl chamber $C$ (hence by the choice of the set of positive roots $\Delta^+$).

\begin{prop}\label{prop::uniquenesstype}
Let $\varphi, \psi \in \mathfrak h_{\mathbf R}^*$ be elements belonging to the dominant Weyl chamber $C$, suppose that both $ih_\varphi, ih_\psi \in \mathfrak g$ have compact isotropy, and that $\varphi'=\lambda \varphi$, $\psi'=\mu \psi$ for some real constants $\lambda, \mu$. Then $\lambda$ and $\mu$ have the same sing. Moreover, $\varphi$ and $\psi$ are multiple each other whenever $\lambda,\mu<0$.
\end{prop}
\begin{proof}
By \eqref{eq::rho=lambdaomegaintermsoftau} we can rewrite $\varphi'=\lambda\varphi$, $\psi'=\mu\psi$ in the form
\begin{equation}
\eta - 2 \sum_{\alpha \in \Delta^+ \cap \varphi^\perp} \alpha = \lambda \varphi,
\qquad
\eta - 2 \sum_{\alpha \in \Delta^+ \cap \psi^\perp} \alpha = \mu \psi.
\end{equation}
By Lemma \ref{lem::lambdaintermsofvarphiandtheta}, taking scalar products with $\psi$ and $\varphi$ respectively yields
\begin{equation}
\mu |\psi|^2 - \lambda (\varphi,\psi) = \sum_{\alpha \in \Delta^+ \cap \varphi^\perp} 2(\alpha,\psi),
\quad
\lambda |\varphi|^2 - \mu (\psi,\varphi) = \sum_{\alpha \in \Delta^+ \cap \psi^\perp} 2(\alpha,\varphi).
\end{equation}
Since $\varphi$ and $\psi$ belong to $C$, one has $(\alpha,\varphi) \geq 0$, $(\alpha,\psi) \geq 0$ for all $\alpha \in \Delta^+$ and $(\varphi,\psi) > 0$.
In particular, the right hand sides of identities above are non-negative.
As a consequence, it is easy to check by contradiction that $\lambda$ and $\mu$ must have the same sign (including zero).

Finally, note that a linear combination of identities above yields
\begin{equation}
\sum_{\alpha \in \Delta^+ \cap \varphi^\perp} 2\mu (\alpha,\psi) 
+ \sum_{\alpha \in \Delta^+ \cap \psi^\perp} 2\lambda (\alpha,\varphi) = |\lambda \varphi - \mu \psi|^2,
\end{equation}
whence it follows that $\lambda \varphi = \mu \psi$ as soon as $\lambda, \mu$ are both negative.
\end{proof}

The uniqueness up to scaling of the solution of the equation $\varphi' = \lambda \varphi$ is a special feature of the case $\lambda <0$.
In case $\lambda>0$, one has just finiteness up to scaling of the set of solutions, in accordance with the next Proposition and the symplectic Fano examples in section \ref{sec::Vogandiagrams}.
Finally, we anticipate here that in case $\lambda=0$ also finiteness up to scaling fails in accordance with Theorem \ref{thm::specialvCY} below.

\begin{prop}\label{prop::finitenessuptoscaling}
Up to scaling, the set of all $\varphi \in \mathfrak h_{\mathbf R}^*$ belonging to the dominant Weyl chamber $C$, such that $v=ih_\varphi$ has compact isotropy, and satisfying $\varphi'=\lambda \varphi$ for some $\lambda \neq 0$ is finite.
\end{prop}
\begin{proof}
Given $\varphi$ as in the statement, after rescaling $\varphi$ to $|\lambda| \varphi$, by \eqref{eq::rho=lambdaomegaintermsoftau} one can write $\varphi' = \lambda \varphi$ in the form
\begin{equation}
\eta - 2 \sum_{\alpha \in \Delta^+ \cap \varphi^\perp} \alpha = \pm \varphi,
\end{equation}
whence it follows that $\varphi$ belongs to the root lattice.
As a consequence, recalling that by compactness of the isotropy group of $|\lambda|v$ no non-compact roots can be orthogonal to $\varphi$, by triangle inequality one has
\begin{equation*}
|\varphi| \leq |\eta| + \sum_{\alpha \in \Delta_c} |\alpha|,
\end{equation*}
where $\Delta_c \subset \Delta$ denotes the subset of compact roots.
Since the right hand side just depends on $\mathfrak g$, we can conclude that $\varphi$ has to belong to a subset of the root lattice which is bounded, hence finite.
\end{proof}

Finiteness up to scaling of the proposition above can be refined in the following result (recall that $\ell$ denotes the rank of $\mathfrak g$).

\begin{thm}\label{thm::specialvFGT}
An element $\varphi \in \mathfrak h_{\mathbf R}^*$ belongs to the dominant Weyl chamber $C$, the stabilizer of $v=ih_\varphi$ is compact, and one has $\varphi'=\lambda \varphi$ for some $\lambda \neq 0$ if and only if there exists a non-empty subset $S \subset \{ 1,\dots,\ell \}$ such that 
\begin{itemize}
\item all non-compact simple roots of $\Sigma^+$ belong to $\{\gamma_i | i \in S\}$, 
\item $\frac{1}{\lambda}\left(\eta - 2 \sum_{\alpha \in \Span\{ \gamma_i | i \in S^c\} \cap \Delta^+} \alpha \right)$ belongs to the open convex cone spanned by $\{\varphi_i | i \in S\}$,
\item $\varphi = \frac{1}{\lambda} \left( \eta - 2 \sum_{\alpha \in \Span\{ \gamma_i | i \in S^c\} \cap \Delta^+} \alpha \right)$.
\end{itemize}
\end{thm}
\begin{proof}
Assume that $\varphi$ belongs to the dominant Weyl chamber $C \subset \mathfrak h_{\mathbf R}^*$, the stabilizer of $v=ih_\varphi$ is compact, and $\varphi'= \lambda \varphi$ for some non-zero $\lambda$.
By the discussion after Lemma \ref{lem::vassumoffunddomweights}, $\varphi$ belongs to the convex cone generated by the fundamental dominant weights $\varphi_1, \dots, \varphi_\ell$.
In particular, one can write $\varphi = \sum_{i=1}^\ell v^i \varphi_i$ for suitable real coefficients $v^1, \dots, v^\ell \geq 0$.
Moreover, the coefficient $v^i$ must be non-zero, hence positive, for any non-compact simple root $\gamma_i$ (otherwise the stabilizer of $v$ would not be compact).
The upshot is that the set $S$, constituted by all indices $1 \leq i \leq \ell$ such that $v^i > 0$, is non-empty, it satisfies the first point of the statement and one has $\varphi = \sum_{i \in S} v^i \varphi_i$.
On the other hand, we have $\varphi' = \eta - 2 \sum_{\alpha \in \Delta^+ \cap \varphi^\perp} \alpha $.
Moreover, note that Lemma \ref{lem::vassumoffunddomweights} imply $\Delta^+ \cap \varphi^\perp = \Span\{\gamma_i | i \in S^c\} \cap \Delta^+$.
Therefore the equation $\varphi' = \lambda \varphi$ can be rewritten as in the third point of the statement, whence the second point follows readily.

Conversely, if $\varphi \in \mathfrak h_0^*$ and there is $S \subset \{1,\dots,\ell\}$ satisfying all three points of the statement for some real $\lambda \neq 0$, then it is easy to check that $(\varphi,\alpha) \geq 0$ for any positive root $\alpha$ (hence $\varphi \in C$), $v=ih_\varphi$ has compact isotropy and $\varphi'=\lambda \varphi$.
\end{proof}

Similarly, the set of all $\varphi$'s satisfying $\varphi'=0$ can be described in terms of certain subsets of $\{1,\dots,\ell\}$. Note that, in contrast with the cases when $\lambda \neq 0$, such a set is very often infinite even up to scaling (more precisely this happens whenever the set $S$ appearing in the statement has cardinality bigger than one).

\begin{thm}\label{thm::specialvCY}
An element $\varphi \in \mathfrak h_{\mathbf R}^*$ belongs to the dominant Weyl chamber $C$, the stabilizer of $v=ih_\varphi$ is compact, and one has $\varphi'=0$ if and only if there exists a non-empty subset $S \subset \{ 1,\dots,\ell \}$ such that 
\begin{itemize}
\item all non-compact simple roots of $\Sigma^+$ belong to $\{\gamma_i | i \in S\}$, 
\item $\varphi = \sum_{i \in S} v^i\varphi_i$ with $v^i >0$ for all $i \in S$, 
\item $\eta - 2 \sum_{\alpha \in \Span\{ \gamma_i | i \in S^c\} \cap \Delta^+} \alpha = 0$.
\end{itemize}
\end{thm}
\begin{proof}
Assume that $\varphi$ belongs to the dominant Weyl chamber $C \subset \mathfrak h_{\mathbf R}^*$, the stabilizer of $v=ih_\varphi$ is compact, and $\varphi'= 0$.
By the discussion after Lemma \ref{lem::vassumoffunddomweights}, $\varphi$ belongs to the convex cone generated by the fundamental dominant weights $\varphi_1, \dots, \varphi_\ell$.
Thereofore $\varphi = \sum_{i=1}^\ell v^i\varphi_i$ for suitable real coefficients $v^1, \dots, v^\ell \geq 0$.
Moreover, the coefficient $v^i$ must be non-zero, hence positive, for any non-compact simple root $\gamma_i$ (otherwise the stabilizer of $v$ would not be compact).
The upshot is that the set $S$ constituted by all index $1 \leq i \leq \ell$ such that $v^i \neq 0$ is non-empty and satisfies the first and second point of the statement.
Finally, note that by \eqref{eq::rho=lambdaomegaintermsoftau} the condition $\varphi'=0$ can be written as
\begin{equation*}
\eta - 2 \sum_{\alpha \in \Delta^+ \cap \Span\{\gamma_i | v^i=0\}} \alpha = 0,
\end{equation*}
and by definition of $S$ it follows that $v^i=0$ if and only if $i$ belongs to the complement of $S$.
Thus the third point of the statement is satisfied as well.

Conversely, if $\varphi \in \mathfrak h_0^*$ and there is $S \subset \{1,\dots,\ell\}$ satisfying all three points of the statement, then it is easy to check that $v=ih_\varphi$ has compact isotropy, $(\varphi,\alpha) \leq 0$ for any positive root $\alpha$ and $\varphi'=0$.
\end{proof}

\section{Hermitian scalar curvature}\label{sec::Hermitianscalarcurvature}

Recall that the Hermitian scalar curvature of an almost complex structure $J$ compatible with a symplectic form $\omega$ on a $n$-dimensional manifold $M$, is the smooth function $s$ defined by $s \omega^n = n \rho \wedge \omega^{n-1}$, being $\rho$ the Chern-Ricci form of $J$.
At a point $p \in M$ one can choose a symplectic basis of the form $e_1,Je_1,\dots,e_n,Je_n$ of $T_pM$ and calculate $s(p) = \sum_{i=1}^n \rho(e_i,Je_i)$.

In our situation, $M=G/V$ is an adjoint orbit and both $\omega$ and $J$ are homogeneous.
Therefore the Hermitian scalar curvature $s$ of $J$ is constant and it is enough to calculate it at the identity coset $e \in G/V$.
To this end, recall that by Corollary \ref{cor::symplecticorthonormalbasis} a symplectic basis of $\mathfrak m$ is given by $(1/\sqrt{\lambda_\alpha})u_\alpha$, $(\varepsilon_\alpha/\sqrt{\lambda_\alpha})v_\alpha$ as $\alpha$ runs in $\Delta^+ \setminus \varphi^\perp$.
On the other hand, we already observed in section \ref{sec::spacialitycondition} that the Chern-Ricci form of $J$ is induced by the two form on $\mathfrak m$ defined by $B(v',[x,y])$, where $v'= 2 \sum_{\alpha \in \Delta^+ \setminus \varphi^\perp} [u_\alpha,v_\alpha]$.
Therefore, by discussion above, the Hermitian scalar curvature of $J$ is given by $s = B(v',z)$, where $z \in \mathfrak g$ is defined by
\begin{equation}
z = \sum_{\alpha \in \Delta^+ \setminus \varphi^\perp}\frac {\varepsilon_\alpha}{\lambda_\alpha} [u_\alpha,v_\alpha].
\end{equation}
By \eqref{eq::uandvintermsofe}, one readily calculates $z = \sum_{\alpha \in \Delta^+ \setminus \varphi^\perp}\frac {-i}{\lambda_\alpha} h_\alpha$ whence one can write $z = ih_\zeta$ where $\zeta \in \mathfrak h_{\mathbf R}^*$ is defined by
\begin{equation}
\zeta = -\sum_{\alpha \in \Delta^+ \setminus \varphi^\perp}\frac {1}{\lambda_\alpha} \alpha.
\end{equation}
On the other hand, recall that by \eqref{eq::v'sumofroots} one has $v' = i h_{\varphi'}$ with $\varphi' = - 2 \sum_{\alpha \in \Delta^+ \setminus \varphi^\perp} \varepsilon_\alpha \alpha$.
Therefore $s = - (\varphi', \zeta)$ 
whence, recalling that $\lambda_\beta = (\varphi,\beta)$ for all $\beta \in \Delta^+$, one readily has the following
\begin{lem}\label{lem::Hermitianscalarcurvature}
The Hermitian scalar curvature of $J$ is given by
\begin{equation}
s = -2 \sum_{\alpha,\beta \in \Delta^+ \setminus \varphi^\perp} \varepsilon_\alpha \frac{(\alpha,\beta)}{(\varphi,\beta)}.
\end{equation}
\end{lem}

Whenever a compatible almost complex structure on a symplectic manifold satisfies $\rho = \lambda \omega$, the Hermitian scalar curvature is given by $s=n\lambda$, where $2n$ is the dimension of the manifold.
Therefore, combining formula above together with Lemma \ref{lem::lambdaintermsofvarphiandtheta} yields the following dimension formula

\begin{prop}
Assume that $\varphi' = \lambda \varphi$.
If $\lambda \neq 0$ then
\begin{equation*}
\dim G/V = - \frac{4|\varphi|^2}{(\eta,\varphi)} \sum_{\alpha,\beta \in \Delta^+ \setminus \varphi^\perp} \varepsilon_\alpha \frac{(\alpha,\beta)}{(\varphi,\beta)}.
\end{equation*}
\end{prop}

\section{Nijenhuis tensor}\label{sec::Nijenhuis}

Recall that the Nijenhuis tensor $N$ of an almost complex structure $J$ on a manifold $M$ is a $TM$-valued two-form defined in terms of commutators of vector fields by the identity 
\begin{equation*}
4N(X,Y) = [JX,JY] - J[JX,Y] - J[X,JY] - [X,Y].
\end{equation*}
The celebrated Newlander-Nirenberg theorem states that $J$ is integrable (i.e. it comes from a complex structure on $M$) if and only if $N=0$ \cite{NewlanderNirenberg1953}.
Thus $N$ constitutes an obstruction to the integrability of $J$.
For future reference, note that $N$ satisfies the identity 
\begin{equation}\label{eq::identitiesN}
N(JX,Y) = -JN(X,Y).
\end{equation}
If $J$ is compatible with a symplectic form $\omega$ on $M$, then the pointwise squared norm $|N|^2$ with respect to the Riemannian metric induced by $J$ defines a smooth function on $M$.
Clearly one can conclude that $J$ is integrable if and only if $|N|^2=0$.
In order to calculate $|N|^2$ at a point $p \in M$ one can choose a symplectic basis of the form $e_1,Je_1,\dots,e_n,Je_n$ of $T_pM$, so that 
\begin{equation}\label{eq::|N|^2(p)}
|N|^2(p) = 2\sum_{i,j=1}^n |N(e_i,e_j)|^2
\end{equation}

In our situation, being $J$ an homogeneous almost complex structure on $G/V$, the Nijenhuis tensor $N$ is homogeneous as well. Therefore, it is completely determined by its value at the identity coset $e \in G/V$, where it can be described as a skew-symmetric bilinear map 
\begin{equation*}
N_H : \mathfrak m \times \mathfrak m \to \mathfrak m.
\end{equation*}
Moreover, since $\omega$ is homogeneous as well, the squared norm $|N|^2$ is a constant function on $G/V$.
In this section we will calculate $N_H$ and $|N|^2$ in terms of roots.

Recall that given $v \in \mathfrak g$ with compact isotropy $K \subset G$, we defined $\varphi \in \mathfrak h_{\mathbf R}^*$ such that $v = ih_\varphi$ and we have chosen positive roots $\Delta^+$ so that $\varphi$ belongs to the dominant Weyl chamber $C$.
Recall moreover that we defined a basis of $\mathfrak m$ (actually of $\mathfrak h_0^\perp$, but here we just need $\mathfrak m$) by letting
\begin{equation}\label{eq::uandvintermsofebis}
u_\alpha = \frac{i^{(1-\varepsilon_\alpha)/2}}{\sqrt 2} (e_\alpha + e_{-\alpha}), \qquad
v_\alpha = \frac{i^{(3-\varepsilon_\alpha)/2}}{\sqrt 2} s_\alpha (e_\alpha - e_{-\alpha})
\end{equation}
for all $\alpha \in \Delta^+ \setminus \varphi^\perp$.
Moreover we defined $H$ by letting $Hu_\alpha = \varepsilon_\alpha v_\alpha$, and $Hv_\alpha = - \varepsilon_\alpha u_\alpha$.
Therefore, $N_H$ is given by (see e.g. \cite[Lemma 12]{DellaVedova2017})
\begin{equation*}
4 N_H(u_\alpha,u_\beta)
= \varepsilon_\alpha \varepsilon_\beta [v_\alpha,v_\beta]_{\mathfrak m}
- \varepsilon_\alpha H [v_\alpha,u_\beta]_{\mathfrak m}
- \varepsilon_\beta H [u_\alpha,v_\beta]_{\mathfrak m}
- [u_\alpha,u_\beta]_{\mathfrak m},
\end{equation*}
where $[x,y]_{\mathfrak m}$ denotes the projection of the commutator $[x,y]$ to $\mathfrak m$.
Moreover, by \eqref{eq::identitiesN} the bilinear map $N_H$ also satisfies
\begin{align*}
N_H(u_\alpha,v_\beta) 
&= -\varepsilon_\beta HN_H(u_\alpha, u_\beta),
&N_H(v_\alpha,v_\beta) 
&= - \varepsilon_\alpha \varepsilon_\beta N_H(u_\alpha,u_\beta).
\end{align*}
As a consequence $N_H$ is entirely determined by $N_H(u_\alpha,u_\beta)$ as $\alpha,\beta \in \Delta^+ \setminus \varphi^\perp$.
After substituting \eqref{eq::uandvintermsofebis} into equation above, a lengthy and uninspired calculation involving Theorem \ref{thm::basisgc} yields
\begin{lem}\label{lem::NH}
For all $\alpha, \beta \in \Delta^+ \setminus \varphi^\perp$, one has
\begin{multline*}
N_H(u_\alpha,u_\beta)
= \frac{(\varepsilon_\alpha+1)(\varepsilon_\beta+1)}{4\sqrt{2}} N_{\alpha,\beta} v_{\alpha+\beta} \\
+ \frac{(\varepsilon_\alpha\varepsilon_\beta - 1)s_{\alpha-\beta}+\varepsilon_\alpha-\varepsilon_\beta}{4\sqrt{2}} N_{\alpha,-\beta} (v_{\alpha-\beta})_{\mathfrak m},
\end{multline*}
where $(v_{\alpha-\beta})_{\mathfrak m}$ denotes the component of $v_{\alpha-\beta}$ along $\mathfrak m$ according to the decomposition $\mathfrak g = \mathfrak v \oplus \mathfrak m$, and the constants $N_{\alpha,\beta}$, $N_{\alpha,-\beta}$ are as in Theorem \ref{thm::basisgc}.
\end{lem}
A few comments about the meaning of this formula are in order.
First of all, note that for all $\alpha,\beta \in \Delta^+ \setminus \varphi^\perp$ the sum $\alpha+\beta$ may not be a root. 
In this case one has $N_{\alpha,\beta}=0$ so that the first summand of r.h.s. vanishes and we do not need to make sense for $v_{\alpha+\beta}$.
On the other hand, whenever $\alpha+\beta$ is a root, it belongs certainly to $\Delta^+ \setminus \varphi^\perp$ for $\varphi$ belongs to the dominant Weyl chamber, and one has
\begin{equation}
(\varphi,\alpha+\beta) 
= (\varphi,\alpha) + (\varphi,\beta) > 0.
\end{equation}
Second, note that for all $\alpha,\beta \in \Delta^+ \setminus \varphi^\perp$ the difference $\alpha-\beta$ is not necessarily a positive root, and $\alpha-\beta$ may well be orthogonal to $\varphi$.
The former is not a big deal, since $N_{\alpha,-\beta}v_{\alpha-\beta}$ is invariant under switching $\alpha$ and $\beta$ and it vanishes whenever $\alpha-\beta$ is not a root.
On the other hand, similarly as above one has
\begin{equation}
(\varphi,\alpha-\beta) 
= (\varphi,\alpha) - (\varphi,\beta),
\end{equation}
whence it follows that one of the roots $\alpha-\beta$ and $\beta-\alpha$ belongs to $\Delta^+ \setminus \varphi^\perp$ precisely when $(\varphi,\alpha) \neq (\varphi,\beta)$, and in this case one has $s_{\alpha-\beta}(\varphi,\alpha-\beta) > 0$.
Therefore, either $(v_{\alpha-\beta})_\mathfrak m=0$ and $(\varphi,\alpha - \beta)=0$, or $(v_{\alpha-\beta})_\mathfrak m = v_{\alpha-\beta}$ and $(\varphi,\alpha - \beta) \neq 0$.

Third, note that $(\varepsilon_\alpha+1)(\varepsilon_\beta+1)=0$ whenever at least one root among $\alpha$ and $\beta$ is compact.
Similarly, $(\varepsilon_\alpha\varepsilon_\beta - 1)s_{\alpha-\beta}+\varepsilon_\alpha-\varepsilon_\beta=0$ precisely when $\alpha$ and $\beta$ are both compact or both non-compact.
Therefore $N_H(u_\alpha,u_\beta) = 0$ if both $\alpha$ and $\beta$ are compact, $N_H(u_\alpha,u_\beta) = \frac{1}{\sqrt{2}} N_{\alpha,\beta} v_{\alpha+\beta}$ if both $\alpha$ and $\beta$ are non-compact, and $N_H(u_\alpha,u_\beta) = \frac{\varepsilon_\alpha - s_{\alpha-\beta}}{2\sqrt{2}}N_{\alpha,-\beta} (v_{\alpha-\beta})_{\mathfrak m}$ otherwise.

At this point we can express the squared norm of $N$ in terms of root data.
To this end, recall that by Corollary \ref{cor::symplecticorthonormalbasis} one has  a symplectic orthonormal basis of $\mathfrak m$ constituted by $(1/\sqrt{\lambda_\alpha}) u_\alpha$, $(\varepsilon_\alpha/ \sqrt{\lambda_\alpha})v_\alpha$ with $\alpha \in \Delta^+ \setminus \varphi^\perp$.
Therefore, by Lemma \ref{lem::NH} and \eqref{eq::|N|^2(p)} one gets the following
\begin{thm}\label{thm::|N|^2}
The squared norm of the Nijenhuis tensor of $J$ is given by
\begin{multline*}
|N|^2 = \sum_{\alpha, \beta \in \Delta_n^+ \setminus \varphi^\perp} \frac{(\varphi,\alpha+\beta)}{(\varphi,\alpha)(\varphi,\beta)} N_{\alpha,\beta}^2 \\
+ \sum_{\alpha \in \Delta_c^+ \setminus \varphi^\perp}
\sum_{\beta \in \Delta_n^+ \setminus \varphi^\perp} (1+s_{\alpha-\beta}) \frac{(\varphi,\alpha-\beta)}{(\varphi,\alpha)(\varphi,\beta)} N_{\alpha,-\beta}^2,
\end{multline*}
where $\Delta_c^+$, $\Delta_n^+$ denote the sets of positive roots that are compact and non-compact respectively, and the constants $N_{\alpha,\beta}$, $N_{\alpha,-\beta}$ are as in Theorem \ref{thm::basisgc}.
\end{thm}
\begin{proof}
By \eqref{eq::|N|^2(p)}, recalling that $\lambda_\alpha = (\varphi,\alpha)$ for all $\alpha \in \Delta^+$, one has
\begin{equation*}
|N|^2 = 2 \sum_{\alpha,\beta \in \Delta^+ \setminus \varphi^\perp} \frac{|N_H(u_\alpha,u_\beta)|^2}{(\varphi,\alpha)(\varphi,\beta)}.
\end{equation*}
By Lemma \ref{lem::NH} then one has
\begin{multline*}
|N|^2
= \sum_{\alpha, \beta \in \Delta_n^+ \setminus \varphi^\perp}  \frac{N_{\alpha,\beta}^2 |v_{\alpha+\beta}|^2}{(\varphi,\alpha)(\varphi,\beta)} \\
+ \sum_{\alpha \in \Delta_c^+ \setminus \varphi^\perp}
\sum_{\beta \in \Delta_n^+ \setminus \varphi^\perp} \frac{(1+s_{\alpha-\beta})^2}{4(\varphi,\alpha)(\varphi,\beta)} N_{\alpha,-\beta}^2 |(v_{\alpha-\beta})_{\mathfrak m}|^2 \\
+ \sum_{\alpha \in \Delta_n^+ \setminus \varphi^\perp}
\sum_{\beta \in \Delta_c^+ \setminus \varphi^\perp} \frac{(1-s_{\alpha-\beta})^2}{4(\varphi,\alpha)(\varphi,\beta)} N_{\alpha,-\beta}^2 |(v_{\alpha-\beta})_{\mathfrak m}|^2,
\end{multline*}
whence, by the identities $N_{\alpha,-\beta} = N_{\beta,-\alpha}$, $v_{\alpha-\beta} = v_{\beta-\alpha}$, $s_{a\alpha-\beta}=-s_{\beta-\alpha}$, and $(1+s_{\alpha-\beta})^2 = 2+2s_{\alpha-\beta}$, it follows
\begin{multline*}
|N|^2 = \sum_{\alpha, \beta \in \Delta_n^+ \setminus \varphi^\perp}  \frac{N_{\alpha,\beta}^2 |v_{\alpha+\beta}|^2}{(\varphi,\alpha)(\varphi,\beta)} \\
+ \sum_{\alpha \in \Delta_c^+ \setminus \varphi^\perp}
\sum_{\beta \in \Delta_n^+ \setminus \varphi^\perp} \frac{1+s_{\alpha-\beta}}{(\varphi,\alpha)(\varphi,\beta)} N_{\alpha,-\beta}^2 |(v_{\alpha-\beta})_{\mathfrak m}|^2 
\end{multline*}
Finally, note that $|v_{\alpha+\beta}|^2=(\varphi,\alpha+\beta)$ and $|(v_{\alpha-\beta})_{\mathfrak m}|^2=s_{\alpha-\beta}(\varphi,\alpha-\beta)$, whence the thesis follows.
\end{proof}

\begin{rem}\label{rem::crtinnonintegrability}
As we already observed, $s_{\alpha-\beta}(\varphi,\alpha-\beta) \geq 0$ for $\varphi$ belongs to the dominant Weyl chamber $C$.
Therefore, all summands of the formula for $|N|^2$ of Theorem \ref{thm::|N|^2} are positive.
In particular, $J$ turns out to be non-integrable as soon as there are two non-compact positive roots $\alpha$, $\beta$ such that $\alpha+\beta$ is a root.
\end{rem}

\section{Compact quotients}\label{sec::compactquotients}

As above, given an element $v \in \mathfrak g$ having compact stabilizer $V \subset G$, consider the adjoint orbit $G/V$ endowed with the Kirillov-Kostant-Souriau symplectic form $\omega$ and the canonical complex structure $J$.  

Our interest here is in a compact locally homogeneous manifold of the form 
\begin{equation*}
M=\Gamma \backslash G/V,
\end{equation*} where $\Gamma \subset G$ is any discrete uniform subgroup, whose existence is guaranteed by a theorem of Borel \cite{Borel1963}.
Since $\omega$ and  $J$ are $G$-invariant, they descend to a symplectic form and an almost complex structure on $M$, which we shall denote by $\omega_\Gamma$ and $J_\Gamma$ respectively.
Therefore $(M,\omega_\Gamma,J_\Gamma)$ is a compact almost K\"ahler manifold.
Dropping for a moment the almost complex structure $J_\Gamma$, we consider the following
\begin{que}\label{que::Kahlerianity}
Does the compact symplectic manifold $(M,\omega_\Gamma)$ admit a (non-necessarily locally homogeneous) compatible complex structure? 
\end{que}
In other words, we are asking if among all almost complex structures compatible with $\omega_\Gamma$ there is one, say $J'$, which is integrable, thus making $(M,\omega_\Gamma,J')$ a K\"ahler manifold.

In this section we aim to give some partial answers to this question.

First of all note that, due to their local nature, all geometric objects such as the Chern-Ricci form, the Hermitian scalar curvature, and the Nijenhuis tensor of $J$ on $(G/V,\omega)$ descend to the corresponding objects on $(M,\omega_\Gamma,J_\Gamma)$.
In particular, the answer to Question \ref{que::Kahlerianity} is positive whenever $J$ is integrable on $G/V$ for $(M,\omega_\Gamma,J_\Gamma)$ is K\"ahler in this case.

On the other hand, in certain circumstances involving just the group $G$, the answer to Question \ref{que::Kahlerianity} is negative.
In fact we have the following result which is due to Carlson and Toledo \cite[Theorem 8.2]{CarlsonToledo1989}.

\begin{thm}\label{thm: noHSnoKahler}
Let $K \subset G$ be a maximal compact subgroup.
If $G/K$ is not Hermitian symmetric, then $(M,\omega_\Gamma)$ is not of the homotopy type of a compact K\"ahler manifold.
\end{thm}

Note that in order to fit with the statement of the mentioned Theorem of Carlson and Toledo one has to drop the symplectic form $\omega_\Gamma$ and to endow $M=G/V$ with a homogeneous complex structure, say $\tilde J$, which always exists due to our hypothesis on $G$ and $V$ \cite[Section 2]{GriffithsSchmid1969}.
The point here is that $\tilde J$ is integrable, but one can check that it is not compatible with $\omega_\Gamma$ (positiveness of the associated pseudo-Riemannian metric fails).

Finally, assume that $J$ satisfies $\rho = \lambda \omega$ on $G/V$.
As a consequence, the same equation is satisfied on $M$, so that one can conclude that the first Chern class $c_1$ of $(M,\omega_\Gamma)$, which can be represented by $\rho_\Gamma/4\pi$ \cite[Section 2]{DellaVedova2017}, satisfies 
\begin{equation*}
4 \pi c_1 = \lambda [\omega_\Gamma] \in H^2_{dR}(M).
\end{equation*}
Therefore, if an adjoint orbit $(G/V,\omega,J)$ satisfies $\rho=\lambda\omega$, then the compact symplectic manifold $(M,\omega_\Gamma)$ turns out to be
\begin{itemize}
\item symplectic general type if $\lambda<0$,
\item symplectic Calabi-Yau if $\lambda=0$,
\item symplectic Fano if $\lambda>0$.
\end{itemize}

In the symplectic general type case, it may be that $J$ itself is integrable.
Therefore $J_\Gamma$ is integrable as well, thus giving several occurrences when the answer to Question \ref{que::Kahlerianity} is positive.
In particular this happens when $G/V$ is Hermitian symmetric.
We refer to Section \ref{sec::Vogandiagrams} for examples when $G$ is simple.

On the other hand one has the following
\begin{lem}\label{lem: sympFanoNotKahler}
If $(M,\omega_\Gamma)$ is symplectic Fano, then it is not of K\"ahler type.
\end{lem}
\begin{proof}
Also in this case the obstruction has topological nature.
Indeed, if $J'$ were an integrable complex structure compatible with $\omega_\Gamma$, then $(M,\omega_\Gamma,J')$ would be a compact K\"ahler manifold with positive Ricci-curvature. As such, by Myers theorem it would have finite fundamental group, contradicting that $M$ is covered by $G/V$, which is non-compact.
\end{proof}

After looking at several examples of adjoint orbits $(G/V,\omega,J)$ with simple group $G$ that can be produced by techniques discussed in Section \ref{sec::Vogandiagrams}, we are pushed to suspect that the answer to Question \ref{que::Kahlerianity} is that a compact quotient $(M,\omega_\Gamma)$ is of K\"ahler type if and only if $J_\Gamma$ is integrable.
This should be compared with a result of Carlson and Toledo \cite[Theorem 0.1]{CarlsonToledo2014}, where they consider a quotient $M=\Gamma \backslash G/V$ as above, take a homogeneous complex structure $\tilde J$ on $G/V$ (which always exists but it is rarely compatible with $\omega$) and establish when the complex manifold $(M,\tilde J_\Gamma)$ admits a compatible K\"ahler metric. 
Unfortunately their approach seems to be hardly adaptable to our situation.
We plan to came back to this question in the future.

\section{Vogan diagrams}\label{sec::Vogandiagrams}

Vogan diagrams are combinatorial devices through which one can classify real semisimple Lie algebras \cite[Chapter VI]{Knapp1996}.
As we will see, they turn out to be a very convenient tool for studying (and hopefully classify) adjoint orbits $(G/V,\omega,J)$ of simple Lie groups satisfying $\rho = \lambda \omega$.

Recall that a Vogan diagram is a Dynkin diagram with some (including no one or all) painted vertices and some pairs of unpainted vertices related by an automorphism of order two that intertwines the elements of the pair.
A vertex is painted whenever the corresponding simple root is non-compact.
On the other hand, paired unpainted vertices correspond to roots which are intertwined by the Cartan involution.
If there is at least one pair of automorphism-related vertices, the diagram is said to have non-trivial automorphism.

It is easy to check that elements of a real semisimple Lie algebra associated with a Vogan diagram having non-trivial automorphism have non-compact isotropy.
In view of our applications, we are thus reduced to consider diagrams with trivial automorphism. 
For this reason, from now on, by a Vogan diagram we simply mean a Dynkin diagram with some painted vertices.

\begin{lem}\label{lem::correspVogandiagramwithv}
Let $G$ be a real semisimple Lie group with Lie algebra $\mathfrak g$, and let $\ell$ be the rank of $\mathfrak g$.
To any $v \in \mathfrak g$ with compact stabilizer, one can associate a Vogan diagram
and a vector $(v^1,\dots,v^\ell) \in \mathbf R^\ell$ with $v^i \geq 0$.
Moreover $v^i>0$ if the $i$-th node of the Vogan diagram is painted.
\end{lem}
\begin{proof}
As we recalled in Section \ref{sec::structureofg}, given $v$ as in the statement, one can always find a Cartan subalgebra $\mathfrak h_0 \subset \mathfrak g$ which contains $v$. 
Therefore, since the Killing-Cartan form $B$ restricts to a positive definite scalar product on $\mathfrak h_{\mathbf R} = i \mathfrak h_0 \subset \mathfrak g_c$, one defines $\varphi \in \mathfrak h_{\mathbf R}^*$ by letting $\varphi(h) = -B(iv,h)$ for all $h \in \mathfrak h_{\mathbf R}$.
Now one can choose a positive root system $\Delta^+ \subset \mathfrak h_{\mathbf R}^*$ so that $\varphi$ belongs to the associated dominant Weyl chamber $C$ or, in other words, $(\varphi,\alpha) \geq 0$ for all $\alpha \in \Delta^+$.
The Weyl chamber also determines the set of positive simple roots $\Sigma^+ \subset \Delta^+$, and the latter splits as $\Sigma^+ = \Sigma^+_c \cup \Sigma^+_n$ where $\Sigma^+_c$ is the set of compact simple roots, and $\Sigma^+_n$ is the set of non-compact ones.

With all this at hand, one can produce a Vogan diagram by taking the Dynkin diagran associated to $\Sigma^+$ and painting all nodes which correspond to elements of $\Sigma^+_n$.

By the assumption that $\rk \mathfrak g = \ell$, one can label simple roots and write $\Sigma^+ = \{\gamma_1,\dots,\gamma_\ell\}$. 
If $A = (A_{ij})$ denotes the associated Cartan matrix, the fundamental dominant weights $\varphi_1,\dots, \varphi_\ell$ are given by 
\begin{equation*}
\varphi_j = \sum_{i=1}^\ell A^{ij}\gamma_i,
\end{equation*}
where $(A^{ij}) = A^{-1}$.
Since $\varphi$ belongs to the dominant Weyl chamber $C$, by Lemma \ref{lem::vassumoffunddomweights}
it follows that 
\begin{equation*}
\varphi = \sum_{i=1}^\ell v^i \varphi_i
\end{equation*}
with $v^i \geq 0$ for all $i = 1,\dots,\ell$, and $v^i>0$ whenever $\gamma_i$ is non-compact by compactness assumption of the stabilizer of $v$.
Thus, considering the vector $(v^i,\dots,v^\ell) \in \mathbf R^\ell$, the thesis follows by the rule for painting the nodes of the diagram.
\end{proof}

Note that one can associate different Vogan diagrams to the same element $v$.
This is clear by the proof of lemma above, where both a Cartan subalgebra $\mathfrak h_0$ containing $v$ and a Weyl chamber $C$ containing $\varphi$ cannot be chosen in a canonical way.
More specifically, one has non-uniqueness of the associated Vogan diagram whenever $\varphi$ belongs to some wall of the Weyl chamber $C$, so that there exists a different Weyl chamber $C'$ which contains $\varphi$.
On the other hand, the vector $(v^1,\dots,v^\ell)$ is uniquely determined once the Vogan diagram is chosen and the simple roots are labelled.

The correspondence established by Lemma \ref{lem::correspVogandiagramwithv} can be reversed. 
As is well known, starting with a connected Dynkin diagram, one can algorithmically find a positive root system $\Delta^+$ for a complex simple Lie algebra.
Similarly, thanks to Corollary \ref{cor::epsilonalphaformula}, starting with a connected Vogan diagram one can find a positive root system $\Delta^+$, and algorithmically determine for each root in $\Delta^+$ whether it is compact or not, for a real simple Lie algebra $\mathfrak g$.
Denote by $C$ the dominant Weyl chamber associated to $\Delta^+$, and choose $\varphi \in C$ of the form $\varphi = \sum_{i=1}^\ell v^i \varphi_i$, where the $\varphi_i$'s are the fundamental dominant weights and $v^i \geq 0$.
Since we agreed to consider only Vogan diagrams with trivial automorphism, the element $v=ih_\varphi \in \mathfrak g$ has compact stabilizer $V$ for the action of any simple Lie group $G$ having Lie algebra $\mathfrak g$. 

\begin{prop}\label{prop::fromvogantoalgebrastabilizer}
Up to isomorphism, the Lie algebra of the stabilizer of $v$ is given by
\begin{equation*}
\mathfrak v = \mathfrak v_1 \oplus \dots \oplus \mathfrak v_r \oplus \mathbf R^m,
\end{equation*}
where $\mathfrak v_1 \oplus \dots \oplus \mathfrak v_r$ is the compact Lie algebra associated to the Dynkin diagram obtained by removing all vertices corresponding to $v^i$'s that are non-zero from the Vogan diagram, and $m$ is the number of removed vertices.
Moreover, $r$ is the number of connected components of the Dynkin diagram and $\mathfrak v_1, \dots, \mathfrak v_r$ are simple.
\end{prop}
\begin{proof}
By  Proposition \ref{prop::decompvplusm} we know that the Lie algebra of the stabilizer of $v$ is given by
\begin{equation*}
\mathfrak v = \mathfrak h_0 \oplus \sum_{\alpha \in \Delta^+ \cap \varphi^\perp} \Span\{u_\alpha, v_\alpha\}.
\end{equation*}
The splitting $\Delta^+ \cap \varphi^\perp = \Delta^+_1 \cup \dots \cup \Delta^+_r$ as union of irreducible positive root systems, induces a decomposition of $\mathfrak v$ as in the statement, where each $\mathfrak v_k$ is the real simple Lie algebra corresponding to $\Delta^+_k$, and $\mathbf R^m$ is the center of $\mathfrak v$.
Note that each $\mathfrak v_k$ is compact for no non-compact root is orthogonal to $\varphi$.

What is left to prove is that the dimension $m$ of the center of $\mathfrak v$ is the same as the number of $v^i$'s that are non-zero.
To this end, note that the center of $\mathfrak v$ is a subalgebra of $\mathfrak h_0$.
Thus, as a vector space we can identify it with a subspace of $\mathfrak h_{\mathbf R}^*$ by mapping each element $z$ of the center to the co-vector $\psi$ defined by $z = ih_\psi$.
Now, by definition \ref{eq::uandvintermsofe} of $u_\alpha$, $v_\alpha$ and item \ref {item::Bhalphah} of Theorem \ref{thm::basisgc} one readily finds
\begin{equation*}
[z,u_\alpha] = (\psi,\alpha) v_\alpha, \qquad [z,v_\alpha] = - (\psi,\alpha) u_\alpha.
\end{equation*}
Since both have to be zero whenever $z$ belongs to the center of $\mathfrak v$, the latter turns out to have the same dimension as the orthogonal complement of $\Delta^+ \cap \varphi^\perp$.
Since $\varphi$ is expressed as a sum of fundamental dominant weights $\varphi_j$'s, which satisfy $(\gamma_i,\varphi_j)=\delta_{ij}$, one can write $\Delta^+ \cap \varphi^\perp = \Span\{\gamma_i \in \Sigma^+ \,|\, v^i=0\} \cap \Delta^+$.
Therefore the orthogonal complement of the set $\Delta^+ \cap \varphi^\perp$ is given by $\Span\{\varphi_j \,|\, v^j \neq 0\}$, whence the thesis follows.
\end{proof}

At this point, we can consider the orbit $G/V$ and equip it with the Kirillov-Kostant-Souriau symplectic form $\omega$ and the canonical compatible almost complex structure $J$.
As discussed in section \ref{sec::spacialitycondition}, deciding whether $J$ satsfies $\rho = \lambda \omega$ or not is reduced to a combinatorial problem on $\varphi$ (cfr. Theorems \ref{thm::specialvFGT} and \ref{thm::specialvCY}) that can be treated at the Lie algebra level.
In order to simplify the statements of results below we give the following
\begin{defn}
We say that $\varphi$ is \emph{$\lambda$-special} (or just \emph{special}) if the orbit $(G/V,\omega,J)$ of $v=ih_\varphi \in \mathfrak g$ satisfies $\rho=\lambda\omega$.
\end{defn}

The upshot is that for any non-compact real simple Lie group $G$, all adjoint orbits $G/V$ endowed with the Kirillov-Kostant-Souriau form $\omega$ and the canonical compatible almost complex structure $J$ which satisfy $\rho=\lambda\omega$ for some constant $\lambda$ can be algorithmically listed (up to isomorphism and scaling).
Indeed, by discussion above this is equivalent to list (up to scaling) all special $\varphi$'s for all possible connected Vogan diagrams.

As one may expect, the number of such $\varphi$'s grows quite fast with the rank $\ell$ of the diagram (i.e. the number of nodes, which corresponds to the rank of the Lie algebra that the Vogan diagram represents).
Nevertheless it is possible to implement the listing algorithm on a computer and get a complete list up to quite large rank.
We did that up to rank $\ell=9$ (thus including all exceptional cases) on a standard PC with computing time of order of $10^3$ seconds for each higher rank algebra.
For the sake of brevity, we list in Appendix \ref{sec::appendix} all special $\varphi$'s for all connected Vogan diagrams up to rank $\ell=4$ together with all $E$-type ones.

Even though the general pattern is still unclear, leaving little hope for a complete classification, some general (non-)existence results can be red off directly from the Vogan diagram.

The first result in this direction says that, under some circumstances, fundamental dominant weights $\varphi_1,\dots,\varphi_\ell$ are special.
\begin{prop}\label{prop::singlenodespecial}
Given a Vogan diagram with a unique painted node, consider the associated set of simple roots $\Sigma^+=\{\gamma_1,\dots,\gamma_\ell\}$.
Assume that the unique painted node corresponds to the simple root $\gamma_p$.
Then $\varphi_p$ is special.
\end{prop}
\begin{proof}
By Lemma \ref{lem::vassumoffunddomweights} one readily gets that $\varphi_p$ belongs to the dominant Weyl chamber $C$ associated to $\Sigma^+$, and the set $\Delta^+ \setminus \varphi_p^\perp$ of all positive roots that are not orthogonal to $\varphi_p$ is constituted by all roots $\alpha = \sum_{k=1}^\ell n^k \gamma_k$ with $n^p>0$.

Note that any simple root $\gamma_i$ with $i\neq p$ is compact.
For any such a simple root, let $\sigma_i$ be the reflection with respect to the hyperplane orthogonal to $\gamma_i$. For any $\psi \in \mathfrak h_{\mathbf R}^*$ one has
\begin{equation*}
(\varphi_p,\sigma_i(\psi)) 
= (\varphi_p,\psi) - 2\frac{(\psi,\gamma_i)}{(\gamma_i,\gamma_i)} (\psi_p,\gamma_i)
= (\varphi_p,\psi) 
\end{equation*}
whence it follows that $\sigma_i(\alpha) \in \Delta^+\setminus\varphi_p^\perp$ for all $\alpha \in \Delta^+\setminus\varphi_p^\perp$.
As discussed in section \ref{sec::spacialitycondition}, the condition $\rho=\lambda\omega$ is equivalent to $\varphi'=\lambda\varphi$, where $\varphi' = - 2 \sum_{\alpha \in \Delta^+ \setminus \varphi^\perp} \varepsilon_\alpha \alpha$.
Note that for all $i\neq p$ one has $\sigma_i(\varphi')=\varphi'$.
Therefore $\varphi'$ is orthogonal to all compact $\gamma_i$, whence it follows that it must be a multiple of $\varphi$.
\end{proof}

At this point we are able to give the following
\begin{proof}[Proof of Theorem \ref{thm::mainintro}]
The thesis follows readily through some calculations from Proposition above, Proposition \ref{prop::fromvogantoalgebrastabilizer} and Lemma \ref{lem::lambdaintermsofvarphiandtheta}. 
\end{proof}

On the other hand, the next Lemma shows that if the canonical almost K\"ahler structure of an adjoint orbit satisfies $\rho=\lambda\omega$, then the sign of $\lambda$ is determined a priori by the Vogan diagram.
Moreover, it gives an effective criterion for deciding whether a given Vogan diagram admits no special $\varphi$.

\begin{lem}\label{lem::naturefromVogandiagram}
Given a Vogan diagram, consider the associated simple roots set $\Sigma^+= \{\gamma_1,\dots,\gamma_\ell\}$, and let $\varphi$ be a $\lambda$-special element of the dominant Weyl chamber.
If $\gamma_i \in \Sigma^+$ is a non-compact simple root, then the scalar product $(\eta,\varphi_i)$ has the same sign as $\lambda$.
\end{lem}
\begin{proof}
First of all recall that $\eta \in \mathfrak h_{\mathbf R}^*$ is defined by
\begin{equation*}
\eta = -2 \sum_{\alpha \in \Delta^+} \varepsilon_\alpha \alpha,
\end{equation*}
and note that it just depends on the Vogan diagram (in particular it does not depend on $\varphi$).
In the current notation, by \eqref{eq::rho=lambdaomegaintermsoftau}, the hypothesis $\rho= \lambda \varphi$ can be rewritten as
\begin{equation}\label{eq::specialityequivalent}
\eta - 2 \sum_{\alpha \in \Span\{ \gamma_j | j \in S^c\} \cap \Delta^+} \alpha = \lambda \varphi,
\end{equation}
where $S \subset \{1,\dots,\ell\}$ is the set of index such that $\varphi = \sum_{i \in S} v^i \varphi_i$ with $v^i>0$.
Note that any $\varphi_i$ with $i\in S$ is orthogonal to all roots $\alpha$ belonging to $\Span\{ \gamma_j | j \in S^c\}$.
On the other hand, any $i$ such that $\gamma_i$ is non-compact must belong to $S$.
The thesis follows by taking the scalar product of both sides of \eqref{eq::specialityequivalent} with $\varphi_i$ as in the statement.
\end{proof}

Finally, we discuss the integrability of $J$ when the adjoint orbit $(G/V,\omega,J)$ satisfies $\rho=\lambda\omega$.
Thanks to Theorem \ref{thm::|N|^2} this reduces to a problem on $\varphi$.
In order to simplify the statements of results below we give the following
\begin{defn}
We say that $\varphi$ is \emph{integrable} if the orbit $(G/V,\omega,J)$ of $v=ih_\varphi \in \mathfrak g$ has integrable canonical almost complex structure $J$.
\end{defn}

Recall that a Vogan diagram is said to be classical if the underlying Dynkin diagram is of type $A_\ell$, $B_\ell$, $C_\ell$, or $D_\ell$.
The next result states that integrable $\varphi$ for such diagrams are very rare.

\begin{thm}\label{thm::mostlynonintegrable}
Given a Vogan diagram of classical type having at least two painted nodes, any $\varphi$ belonging to the associated Weyl chamber is not integrable.
\end{thm}
\begin{proof}
Notice that the integrability of $\varphi$ depends quite weakly from $\varphi$ itself. Indeed, by Theorem \ref{thm::|N|^2} the norm of the Nijenhuis tensor of the canonical almost complex structure of the orbit of $v=ih_\varphi$ is non-zero as soon as there are two non-commuting non-compact roots. We are then reduced to show that classical Vogan diagrams with at least two painted nodes always admit a couple of non-commuting non-compact roots. The proof of this fact is done case-by-case. The explicit form of the roots is in \cite[Section~12]{Humphreys}
\begin{itemize}
\item $A_{\ell}$: the roots of $A_{\ell}$ are $\Delta=\lrbr{\varepsilon_i - \varepsilon_j\ \vert\ i\neq j}$, where $\varepsilon_1,\ldots,\varepsilon_{\ell}$ form the usual unit basis of $\mathbb{R}^{\ell}$. In particular, $\lrbr{\varepsilon_i-\varepsilon_{i+1}\ \vert 1\leq i\leq \ell}$ form a basis for $\Delta$. Hence, if we put $\gamma_i:=\varepsilon_i - \varepsilon_{i+1}$, we have that $\gamma_i$'s are just the simple roots for $\Delta$. 

Let $P=\lrbr{i_1,\ldots, i_m}$ be the (ordered) set of the painted nodes of the Vogan diagram. Put
\begin{equation*}
\alpha:=\sum_{j=i_1}^{i_2-1}\gamma_j=\sum_{j=i_1}^{i_2-1}\varepsilon_j-\varepsilon_{j+1}=\varepsilon_{i_1}-\varepsilon_{i_2}\in\Delta^+
\end{equation*}
and observe that $\alpha$ is non-compact since the sum of the coefficients relative to non-compact simple roots is just the coefficient of $\gamma_{i_1}$ which is 1 and so odd.
Then the linear combination
\begin{equation*}
\beta=\sum_{j=1_1}^{i_2}\gamma_j=\sum_{j=i_1}^{i_2}\varepsilon_j - \varepsilon_{j+1}=\varepsilon_{i_1}-\varepsilon_{i_2+1}\in \Delta^+
\end{equation*}
gives a root. Moreover $\beta=\alpha+\gamma_{i_2}$. This shows that $[\mathfrak{g}_{\alpha},\mathfrak{g}_{\gamma_{i_2}}]\subseteq\mathfrak{g}_{\beta}$ and $[\alpha,\gamma_{i_2}]\neq 0$. Thus the commutator of the non-compact roots $\alpha$ and $\gamma_{i_2}$ is non-zero.
\item $B_{\ell}$: the roots of $B_{\ell}$ are $\Delta=\lrbr{\pm\varepsilon_k,\ \pm(\varepsilon_i\pm\varepsilon_j)\ \vert\ i\neq j}$ and a basis is given by $\Sigma^+=\lrbr{\varepsilon_i-\varepsilon_{i+1},\varepsilon_{\ell}\ \vert\ 1\leq i\leq \ell-1}$. As above, put $\gamma_i=\varepsilon_i - \varepsilon_{i+1}$, $1\leq i\leq \ell-1$, $\gamma_{\ell}=\varepsilon_{\ell}$ and let $P=\lrbr{i_1,\ldots, i_m}$ be the (ordered) set of painted nodes of the Vogan diagram. We have to distinguish two cases.
\begin{itemize}
\item $i_1,i_2\in\lrbr{1,\ldots, \ell-1}$: it holds the same argument as for $A_{\ell}$;
\item $i_1\in\lrbr{1,\ldots,\ell-1},\ i_2=\ell$: put 
\begin{equation*}
\alpha=\sum_{j=i_1}^{\ell-1}\gamma_j=\sum_{j=i_2}^{\ell-1}\varepsilon_j - \varepsilon_{j+1}=\varepsilon_{i_1}-\varepsilon_{\ell}\in \Delta^+
\end{equation*}
and observe that $\alpha$ is non-compact. Then 
\begin{equation*}
\beta=\alpha+\gamma_{\ell}=\varepsilon_{i_1}-\varepsilon_{\ell} +\varepsilon_{\ell}=\varepsilon_{i_1}\in\Delta^+,
\end{equation*}
so, as above, $[\alpha,\gamma_{\ell}]\neq 0$.
\end{itemize}
\item $C_{\ell}$: the roots of $C_{\ell}$ are $\Delta=\lrbr{\pm 2\varepsilon_k,\ \pm(\varepsilon_i\pm\varepsilon_j)\ \vert\ i\neq j}$ and a basis is given by $\Sigma^+=\lrbr{\varepsilon_i-\varepsilon_{i+1},2\varepsilon_{\ell}\ \vert\ 1\leq i\leq \ell-1}$. As above, put $\gamma_i=\varepsilon_i - \varepsilon_{i+1}$, $1\leq i\leq \ell-1$, $\gamma_{\ell}=2\varepsilon_{\ell}$ and let $P=\lrbr{i_1,\ldots, i_m}$ be the (ordered) set of painted nodes of the Vogan diagram. We have to distinguish two cases.
\begin{itemize}
\item $i_1,i_2\in\lrbr{1,\ldots, \ell-1}$: it holds the same argument as for $A_{\ell}$;
\item $i_1\in\lrbr{1,\ldots,\ell-1},\ i_2=\ell$: put 
\begin{equation*}
\alpha=\sum_{j=i_1}^{\ell-1}\gamma_j=\sum_{j=i_2}^{\ell-1}\varepsilon_j - \varepsilon_{j+1}=\varepsilon_{i_1}-\varepsilon_{\ell}\in \Delta^+
\end{equation*}
and observe that $\alpha$ is non-compact. Then 
\begin{equation*}
\beta=\alpha+\gamma_{\ell}=\varepsilon_{i_1}-\varepsilon_{\ell} +2\varepsilon_{\ell}=\varepsilon_{i_1}+\varepsilon_{\ell}\in\Delta^+,
\end{equation*}
so, $[\alpha,\gamma_{\ell}]\neq 0$.
\end{itemize}
\item $D_{\ell}$: the roots of $D_{\ell}$ are $\Delta=\lrbr{\pm(\varepsilon_i\pm\varepsilon_j)\ \vert\ i\neq j}$ and a basis is given by $\Sigma^+=\lrbr{\varepsilon_i-\varepsilon_{i+1}, \varepsilon_{\ell-1}+\varepsilon_{\ell}\vert\ 1\leq i\leq \ell-1}$. As above, put $\gamma_i=\varepsilon_i - \varepsilon_{i+1}$, $1\leq i\leq \ell-1$, $\gamma_{\ell}=\varepsilon_{\ell-1}+\varepsilon_{\ell}$ and let $P=\lrbr{i_1,\ldots, i_m}$ be the (ordered) set of painted nodes of the Vogan diagram. We have to distinguish various cases.
\begin{itemize}
\item $i_1,i_2\in\lrbr{1,\ldots, \ell-1}$: it holds the same argument as for $A_{\ell}$;
\item $i_1\in\lrbr{1,\ldots,\ell-2},\ i_2=\ell-1$: put 
\begin{equation*}
\alpha=\sum_{j=i_1}^{\ell-2}\gamma_j=\sum_{j=i_2}^{\ell-2}\varepsilon_j - \varepsilon_{j+1}=\varepsilon_{i_1}-\varepsilon_{\ell-1}\in \Delta^+
\end{equation*}
and observe that $\alpha$ is non-compact. Then 
\begin{equation*}
\beta=\alpha+\gamma_{\ell-1}=\varepsilon_{i_1}-\varepsilon_{\ell-1} +\varepsilon_{\ell-1}-\varepsilon_{\ell}=\varepsilon_{i_1}-\varepsilon_{\ell}\in\Delta^+,
\end{equation*}
so, as above, $[\alpha,\gamma_{\ell-1}]\neq 0$;
\item $i_1\in\lrbr{1,\ldots,\ell-2},\ i_2=\ell$: put 
\begin{equation*}
\alpha=\sum_{j=i_1}^{\ell-2}\gamma_j=\sum_{j=i_2}^{\ell-2}\varepsilon_j - \varepsilon_{j+1}=\varepsilon_{i_1}-\varepsilon_{\ell-1}\in \Delta^+
\end{equation*}
and observe that $\alpha$ is non-compact. Then 
\begin{equation*}
\beta=\alpha+\gamma_{\ell}=\varepsilon_{i_1}-\varepsilon_{\ell-1} +\varepsilon_{\ell-1}+\varepsilon_{\ell}=\varepsilon_{i_1}+\varepsilon_{\ell}\in\Delta^+,
\end{equation*}
so, as above, $[\alpha,\gamma_{\ell}]\neq 0$;
\item $(i_1,i_2)=(\ell-1,\ell-2)$: put 
\begin{equation*}
\alpha=\gamma_{\ell-2}+\gamma_{\ell}=\varepsilon_{\ell-2}-\varepsilon_{\ell-1}+\varepsilon_{\ell-1}+\varepsilon_{\ell}=\varepsilon_{\ell-2}+\varepsilon_{\ell}\in\Delta^+
\end{equation*}
and observe that $\alpha$ is non-compact. Then 
\begin{equation*}
\beta=\alpha+\gamma_{\ell-1}=\varepsilon_{\ell-2}+\varepsilon_{\ell} +\varepsilon_{\ell-1}-\varepsilon_{\ell}=\varepsilon_{\ell-2}+\varepsilon_{\ell-1}\in\Delta^+,
\end{equation*}
so, as above, $[\alpha,\gamma_{\ell}]\neq 0$;
\end{itemize}
\end{itemize}
Summing up, for any classical Vogan diagram with at least 2 painted nodes, we found a couple of non commuting non-compact roots.
\end{proof}

As a consequence of the theorem above, only Vogan diagrams of classical type having one painted node may admit integrable $\varphi$'s. 
If $\varphi$ is also special, we can say which they are.
Indeed we have the following
\begin{thm}\label{thm::integrabilityonenode}
Given a Vogan diagram of classical type having just one painted node, assume that $\varphi$ belongs to the associated Weyl chamber and it is special and integrable.
Then the Vogan diagram is one of the following: all diagrams of type $A_{\ell}$,
\begin{center}
\begin{tikzpicture}[scale=0.6]
     \node[anchor=south] at (-50mm,-5mm) {$B_{\ell}$};
   \draw[thick] (-24mm ,0) circle (1mm);
    \node[anchor=south] at (-24mm,1mm) {$1$};
    \node[anchor=north] at (-24mm,-1mm) {$\gamma_1$};
    \draw[thick] (-8mm,0) circle (1mm);
    \node[anchor=south] at (-8mm,1mm) {$2$};
    \node[anchor=north] at (-8mm,-1mm) {$\gamma_2$};
    \draw[thick] (8mm,0) circle (1mm);
    \node[anchor=south] at (8mm,1mm) {$2$};
    \node[anchor=north] at (8mm,-1mm) {$\gamma_3$};
    \draw[thick] (24mm,0) circle (1mm);
    \node[anchor=south] at (24mm,1mm) {$2$};
    \node[anchor=north] at (24mm,-1mm) {$\gamma_{\ell-2}$};
    \draw[thick] (40mm,0) circle (1mm);
    \node[anchor=south] at (40mm,1mm) {$2$};
    \node[anchor=north] at (40mm,-1mm) {$\gamma_{\ell-1}$};
    \draw[thick,fill=black] (56mm,0) circle (1mm);
    \node[anchor=south] at (56mm,1mm) {$2$};
    \node[anchor=north] at (56mm,-1mm) {$\gamma_{\ell}$};
    \draw[thick] (-22mm,0.4mm) -- (-10mm,0.4mm);
    \draw[thick] (-22mm,-0.4mm) -- (-10mm,-0.4mm);
    \draw[thick] (-6mm,0) -- (6mm,0);
    \draw[dotted] (10mm,0) -- (22mm,0);
    \draw[thick] (26mm,0) -- (38mm,0);
    \draw[thick] (42mm,0) -- (54mm,0);
\end{tikzpicture}
\begin{tikzpicture}[scale=0.6]
     \node[anchor=south] at (-50mm,-5mm) {$C_{\ell}$};
   \draw[thick] (-24mm ,0) circle (1mm);
    \node[anchor=south] at (-24mm,1mm) {$1$};
    \node[anchor=north] at (-24mm,-1mm) {$\gamma_1$};
    \draw[thick] (-8mm,0) circle (1mm);
    \node[anchor=south] at (-8mm,1mm) {$1$};
    \node[anchor=north] at (-8mm,-1mm) {$\gamma_2$};
    \draw[thick] (8mm,0) circle (1mm);
    \node[anchor=south] at (8mm,1mm) {$1$};
    \node[anchor=north] at (8mm,-1mm) {$\gamma_3$};
    \draw[thick] (24mm,0) circle (1mm);
    \node[anchor=south] at (24mm,1mm) {$1$};
    \node[anchor=north] at (24mm,-1mm) {$\gamma_{\ell-2}$};
    \draw[thick] (40mm,0) circle (1mm);
    \node[anchor=south] at (40mm,1mm) {$1$};
    \node[anchor=north] at (40mm,-1mm) {$\gamma_{\ell-1}$};
    \draw[thick,fill=black] (56mm,0) circle (1mm);
    \node[anchor=south] at (56mm,1mm) {$2$};
    \node[anchor=north] at (56mm,-1mm) {$\gamma_{\ell}$};
    \draw[thick] (-22mm,0mm) -- (-10mm,0mm);
    \draw[thick] (-6mm,0mm) -- (6mm,0mm);
    \draw[dotted] (10mm,0mm) -- (22mm,0mm);
    \draw[thick] (26mm,0mm) -- (38mm,0mm);
    \draw[thick] (42mm,-0.4mm) -- (54mm,-0.4mm);
    \draw[thick] (42mm,0.4mm) -- (54mm,0.4mm);
\end{tikzpicture}
\begin{tikzpicture}[scale=0.6]
     \node[anchor=south] at (-50mm,-5mm) {$D_{\ell}$};
   \draw[thick,fill=black] (-24mm ,0) circle (1mm);
    \node[anchor=north] at (-24mm,-1mm) {$\gamma_1$};
    \draw[thick] (-8mm,0) circle (1mm);
    \node[anchor=north] at (-8mm,-1mm) {$\gamma_2$};
    \draw[thick] (8mm,0) circle (1mm);
    \node[anchor=north] at (8mm,-1mm) {$\gamma_3$};
    \draw[thick] (24mm,0) circle (1mm);
    \node[anchor=north] at (24mm,-1mm) {$\gamma_{\ell-3}$};
    \draw[thick] (40mm,0) circle (1mm);
    \node[anchor=north] at (40mm,-1mm) {$\gamma_{\ell-2}$};
    \draw[thick] (56mm,14mm) circle (1mm);
    \node[anchor=north] at (65mm,18mm) {$\gamma_{\ell-1}$};
    \draw[thick] (56mm,-14mm) circle (1mm);
    \node[anchor=north] at (63mm,-10mm) {$\gamma_{\ell}$};
    \draw[thick] (-22mm,0mm) -- (-10mm,0mm);
    \draw[thick] (-6mm,0mm) -- (6mm,0mm);
    \draw[dotted] (10mm,0mm) -- (22mm,0mm);
    \draw[thick] (26mm,0mm) -- (38mm,0mm);
    \draw[thick] (42mm,1mm) -- (54.5mm,12.5mm);
    \draw[thick] (42mm,-1mm) -- (54.5mm,-12.5mm);
\end{tikzpicture}
\begin{tikzpicture}[scale=0.6]
     \node[anchor=south] at (-50mm,-5mm) {};
   \draw[thick] (-24mm ,0) circle (1mm);
    \node[anchor=north] at (-24mm,-1mm) {$\gamma_1$};
    \draw[thick] (-8mm,0) circle (1mm);
    \node[anchor=north] at (-8mm,-1mm) {$\gamma_2$};
    \draw[thick] (8mm,0) circle (1mm);
    \node[anchor=north] at (8mm,-1mm) {$\gamma_3$};
    \draw[thick] (24mm,0) circle (1mm);
    \node[anchor=north] at (24mm,-1mm) {$\gamma_{\ell-3}$};
    \draw[thick] (40mm,0) circle (1mm);
    \node[anchor=north] at (40mm,-1mm) {$\gamma_{\ell-2}$};
    \draw[thick,fill=black] (56mm,14mm) circle (1mm);
    \node[anchor=north] at (65mm,18mm) {$\gamma_{\ell-1}$};
    \draw[thick] (56mm,-14mm) circle (1mm);
    \node[anchor=north] at (63mm,-10mm) {$\gamma_{\ell}$};
    \draw[thick] (-22mm,0mm) -- (-10mm,0mm);
    \draw[thick] (-6mm,0mm) -- (6mm,0mm);
    \draw[dotted] (10mm,0mm) -- (22mm,0mm);
    \draw[thick] (26mm,0mm) -- (38mm,0mm);
    \draw[thick] (42mm,1mm) -- (54.5mm,12.5mm);
    \draw[thick] (42mm,-1mm) -- (54.5mm,-12.5mm);
\end{tikzpicture}
\end{center}
Moreover,  $\varphi$ coincides up to scaling with  the fundamental dominant weight corresponding to the painted node. 
\end{thm}
\begin{proof}
Assume that the unique painted node is the $p$-th one and the rank of the Lie algebra is $\ell$. Since $\varphi$ belongs to the dominant Weyl chamber, it is a non-negative linear combination of the fundamental dominant weights $\varphi=\sum_{i=1}^{\ell}v^i\varphi_i$ with $v^p>0$.
Suppose that $\varphi$ is not a multiple of $\varphi_p$.
Since $\varphi$ is special, Proposition \ref{prop::uniquenesstype} and Theorem \ref{thm::specialvCY} force the corresponding orbit to be symplectic Fano, contradicting the integrability condition by Lemma \ref{lem: sympFanoNotKahler}. Thus $\varphi$ is a multiple of $\varphi_p$.
Finally, consider a maximal compact subgroup $K$ of $G$.
Note that by Theorem \ref{thm: noHSnoKahler}, the integrability of the canonical almost complex structure of the orbit forces $G/K$ to be Hermitian symmetric. The thesis follows by the fact that the associated Vogan diagrams are classified \cite[Appendix C.3]{Knapp1996}.
\end{proof}


The previous theorems say which classical Vogan diagrams admit integrable special $\varphi$. In what follows, we treat the exceptional cases. 

Let $\mathfrak{g}$ be a non-compact real exceptional Lie algebra with trivial automorphism, i.e., one among 
\begin{equation*}
\mathfrak{g}_{2(2)},\ \mathfrak{f}_{4(4)},\ \mathfrak{f}_{4(-20)},\ \mathfrak{e}_{6(2)},\ \mathfrak{e}_{6(-14)},\ \mathfrak{e}_{7(7)},\ \mathfrak{e}_{7(-5)},\ \mathfrak{e}_{7(-25)},\ \mathfrak{e}_{8(8)},\  \mathfrak{e}_{8(-24)},
\end{equation*}
where we follow the notation as in \cite{Helgason1978}.
Given an element $v\in\mathfrak{g}$ having compact stabilizer $V\subset G$, define $\varphi$ by $v=ih_{\varphi}$ and suppose it is integrable. Then the canonical complex structure $J$ on $G/V$ descends to an integrable $J_{\Gamma}$ on $M=\Gamma\backslash G/V$, for $\Gamma\subset G$ a uniform lattice making $M$ a K\"ahler manifold. Thus, in force of theorem \ref{thm: noHSnoKahler}, the quotient $G/K$ with $K\subset G$ a maximal compact subgroup has to be Hermitian symmetric. In particular, the only possibilities for $\mathfrak{g}$ are $\mathfrak{g}=\mathfrak{e}_{6(-14)}$ and $\mathfrak{g}=\mathfrak{e}_{7(-25)}$ \cite[Chapter X, table V]{Helgason1978}. 
By the Borel and de Siebenthal theorem \cite[Chapter~VI]{Knapp1996}, in order to find exceptional Vogan diagrams admitting special and integrable $\varphi$, it suffices to look at the ones equivalent to the Vogan diagrams of $\mathfrak{e}_{6(-14)}$ and $\mathfrak{e}_{7(-25)}$. To do this, we look at the tables of $E_6$ and $E_7$ in Appendix \ref{sec::appendix}. 
\begin{itemize}
\item $\mathfrak{e}_{6(-14)}$: there are three Vogan diagrams equivalent to the one of $\mathfrak{e}_{6(-14)}$. First, the diagram with $\gamma_2$ and $\gamma_4$ painted is excluded form our analysis since the associated orbit is symplectic Fano and thus it cannot admit integrable complex structure by Lemma \ref{lem: sympFanoNotKahler}. For the diagram with $\gamma_1$ and $\gamma_5$ painted, observe that the roots
\begin{equation}
\alpha=\gamma_1+\gamma_2+\gamma_3+\gamma_4+\gamma_6,\qquad \beta=\gamma_5
\end{equation}
are both non compact and their sum $\sum_{i=1}^6 \gamma_i$ is a root. Thus we found two non-commuting non-compact roots and by Theorem \ref{thm::|N|^2} the canonical almost complex structure cannot be integrable. We are left with the Vogan diagram of $\mathfrak{e}_{6(-14)}$ which is symplectic general type. Uniqueness of the special vector then follows by Theorem \ref{prop::singlenodespecial} and Proposition \ref{prop::uniquenesstype}.
\item $\mathfrak{e}_{7(-25)}$: in this case there is just one diagram to consider.
Uniqueness of the special vector then follows by Theorem \ref{prop::singlenodespecial} and Proposition \ref{prop::uniquenesstype} for the Vogan diagram being symplectic general type.
\end{itemize}
Summarizing, we have proved the following

\begin{thm}
Given a Vogan diagram of exceptional type, assume that $\varphi$ belongs to the associated Weyl chamber and it is special and integrable.
Then the Vogan diagram is one of the following
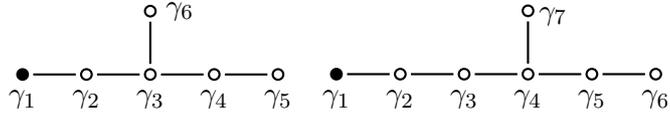
\begin{center}
\begin{tikzpicture}[scale=0.7]
    \draw[thick,fill=black] (0,0) circle (1mm);
    \node[anchor=north] at (0,-1mm) {$\gamma_1$};
    \draw[thick] (12mm,0) circle (1mm);
    \node[anchor=north] at (12mm,-1mm) {$\gamma_2$};
    \draw[thick] (24mm,0) circle (1mm);
    \node[anchor=north] at (24mm,-1mm) {$\gamma_3$};
    \draw[thick] (36mm,0) circle (1mm);
    \node[anchor=north] at (36mm,-1mm) {$\gamma_4$};
    \draw[thick] (48mm,0) circle (1mm);
    \node[anchor=north] at (48mm,-1mm) {$\gamma_5$};
    \draw[thick] (24mm,12mm) circle (1mm);
    \node[anchor=west] at (25mm,12mm) {$\gamma_6$};
    \draw[thick] (2mm,0) -- (10mm,0);
    \draw[thick] (14mm,0) -- (22mm,0);
    \draw[thick] (26mm,0) -- (34mm,0);
    \draw[thick] (38mm,0) -- (46mm,0);
    \draw[thick] (24mm,2mm) -- (24mm,10mm);
\end{tikzpicture}
\begin{tikzpicture}[scale=0.7]
    \draw[thick,fill=black] (0,0) circle (1mm)
    node[anchor=north] at +(0,-1mm) {$\gamma_1$};
    \draw[thick] (12mm,0) circle (1mm)
    node[anchor=north] at +(0,-1mm) {$\gamma_2$};
    \draw[thick] (24mm,0) circle (1mm)
    node[anchor=north] at +(0,-1mm) {$\gamma_3$};
    \draw[thick] (36mm,0) circle (1mm)
    node[anchor=north] at +(0,-1mm) {$\gamma_4$};
    \draw[thick] (48mm,0) circle (1mm)
    node[anchor=north] at +(0,-1mm) {$\gamma_5$};
    \draw[thick] (60mm,0) circle (1mm)
    node[anchor=north] at +(0,-1mm) {$\gamma_6$};
    \draw[thick] (36mm,12mm) circle (1mm)
    node[anchor=west] at +(0,-1mm) {$\gamma_7$};
    \draw[thick] ++(2mm,0) -- ++(8mm,0)
    ++(4mm,0) -- ++(8mm,0)
    ++(4mm,0) -- ++(8mm,0)
    ++(2mm,2mm) -- ++(0,8mm)
    ++(2mm,-10mm) -- ++(8mm,0)
    ++(4mm,0) -- ++(8mm,0);
\end{tikzpicture}
\end{center}
Moreover,  $\varphi$ coincindes up to scaling with  the fundamental dominant weight corresponding to the painted node. 
\end{thm}


\pagebreak

\appendix
\section{Vogan diagrams with special $\varphi$}\label{sec::appendix}

Following notation and terminology of Section \ref{sec::Vogandiagrams}, in this Appendix we classify all special $\varphi$ associated to connected Vogan diagrams of rank at most $\ell=4$ or exceptional.
This classification comes from results discussed in previous sections.
In particular, we have the following

\begin{prop}
Let $G$ be a non-compact real simple Lie group either of rank $\ell \leq 4$ or of exceptional type, and let
$(G/V,\omega,J)$ be an adjoint orbit of $G$ endowed with the canonical almost K\"ahler structure.
If $(G/V,\omega,J)$ satisfies $\rho=\lambda\omega$ then it is isomorphic up to scaling to the orbit of $v=ih_\varphi$ for some $\varphi$ contained in the following tables.
\end{prop}

For each simple Lie algebra type, we list the fundamental dominant weights $\varphi_1,\dots,\varphi_\ell$ in terms of the simple roots $\gamma_1,\dots,\gamma_\ell$.
For each Vogan diagram we list all special $\varphi$ expressed as a sum of dominant weights $\varphi_i$'s, and for each of them we give the following data:
\begin{itemize}
\item whether $\varphi$ is a root,
\item the symplectic type of the orbit of $v$ and the integrability of the canonical complex structure $J$ (we write sGT, sCY or sF if the orbit satisfy $\rho = \lambda \omega$ with $\lambda<0$, $\lambda=0$, $\lambda>0$ respectively, and removing the 's' when when $J$ is integrable),
\item the Hermitian scalar curvature $s$ of the canonical complex structure $J$ (note that from this one can readily calculate $\lambda$ form the identity $s = \lambda \dim G/V$),
\item the dimensions of the stabilizer $V$ of $v$ and of the orbit $G/V$,
\item the Lie algebras of $G$ and $V$.
\end{itemize}

The choice of the acronyms at the second point is motivated as follows.
Given an adjoint orbit $(G/V,\omega,J)$ satisfying $\rho=\lambda \omega$, and a uniform lattice $\Gamma \subset G$, the quotient $(\Gamma\backslash G/V,\omega_\Gamma,J_\Gamma)$ is a compact almost K\"ahler manifold satisfying $\rho_\Gamma = \lambda \omega_\Gamma$ as soon as $\Gamma\backslash G/V$ is smooth. Therefore, as discussed in Section \ref{sec::compactquotients}, $(\Gamma\backslash G/V,\omega_\Gamma)$ is symplectic General type, Calabi-Yau or Fano according to the sign of $\lambda$.
In other words, sGT, sCY, sF denote the symplectic type of any compact quotient of $(G/V,\omega)$.

In order to determine which Lie algebra is associated to a Vogan diagram with more than one painted node, we  used the rules discussed in \cite{ChuahHu2004}.

As a matter of notation, we largely follow Knapp \cite{Knapp1996}.
Just for exceptional simple Lie algebras we follow Helgason notation \cite{Helgason1978}.
For convenience of the reader, we specify the relationship between the two notations in the notes of each exceptional Lie algebra type.

\small

\newcommand{\VoganDiag}[2]
{
\begin{tikzpicture}[scale=0.7]
    \draw[thick,fill={#1}] (0 ,0) circle (1mm);
    \node[anchor=south] at (0,1mm) {$1$};
    \node[anchor=north] at (0,-1mm) {$\gamma_1$};
    \draw[thick,fill={#2}] (12mm,0) circle (1mm);
    \node[anchor=south] at (12mm,1mm) {$1$};
    \node[anchor=north] at (12mm,-1mm) {$\gamma_2$};
    \draw[thick] (2mm,0) -- (10mm,0);
\end{tikzpicture}
}
\begin{center}
\hspace*{-30mm}
\begin{tabular}{|*9{c|}}
\hline
\Huge $A_2$ & \multicolumn{5}{|c|}{
\begin{minipage}[b]{0.6\columnwidth}
\smallskip
$\dim \mathfrak g = 8$. One non-compact simple real form with trivial automorphism: $\mathfrak{su}(1,2)$.
\end{minipage}
} & \multicolumn{3}{l|}{
\begin{minipage}[b]{0.4\columnwidth}
\smallskip
\raggedright Fundamental dominant weights \\ 
$
\begin{array}{lll}
\varphi_1 &=& \frac{2}{3}\gamma_1 + \frac{1}{3}\gamma_2 \\ 
\varphi_2 &=& \frac{1}{3}\gamma_1 + \frac{2}{3}\gamma_2 \\ 
\end{array}
$
\end{minipage}} \\
\hline
\hline
Vogan diagram & $\varphi$ & $\varphi \in \Delta$ & Type & s & $\dim V$ & $\dim G/V$ & $\mathfrak g$ & $\mathfrak v$ \\
\hline
\VoganDiag{black}{white}
& $ \varphi_1 $ & no & GT & $-12$ & $4$ & $4$ & $\mathfrak{su}(1,2)$ & $\mathfrak{su}(2) \oplus \mathbf R$ \\
\hline
\VoganDiag{black}{black}
& \shortstack{$t_1 \varphi_1 + t_2 \varphi_2$ \\ \\ for all $t_1,t_2>0$}
& no & sCY & $0$ & 2 & 6 & $\mathfrak{su}(1,2)$ & $\mathbf R \oplus \mathbf R$ \\
\hline
\end{tabular}
\end{center}

\vspace{1cm}

\renewcommand{\VoganDiag}[2]
{
\begin{tikzpicture}[scale=0.7]
    \draw[thick,fill={#1}] (0,0) circle (1mm);
    \node[anchor=south] at (0,1mm) {$1$};
    \node[anchor=north] at (0,-1mm) {$\gamma_1$};
    \draw[thick,fill={#2}] (12mm,0) circle (1mm);
    \node[anchor=south] at (12mm,1mm) {$2$};
    \node[anchor=north] at (12mm,-1mm) {$\gamma_2$};
    \draw[thick] (2mm,0.4mm) -- (10mm,0.4mm);
    \draw[thick] (2mm,-0.4mm) -- (10mm,-0.4mm);
\end{tikzpicture}
}
\begin{center}
\hspace*{-30mm}
\begin{tabular}{|*9{c|}}
\hline
\Huge $B_2$ & \multicolumn{5}{|c|}{
\begin{minipage}[b]{0.5\columnwidth}%
\smallskip
$\dim \mathfrak g = 10$. Two simple non-compact real forms: $\mathfrak{so}(4,1)$, $\mathfrak{so}(2,3)$.
\end{minipage}
} & \multicolumn{3}{l|}{
\begin{minipage}[b]{0.5\columnwidth}
\smallskip
\raggedright Fundamental dominant weights \\ 
$
\begin{array}{lll}
\varphi_1 &= \gamma_1 + \frac{1}{2} \gamma_2 \\ 
\varphi_2 &= \gamma_1 + \gamma_2 \\ 
\end{array}
$
\end{minipage}} \\
\hline
\hline
Vogan diagram & $\varphi$ & $\varphi \in \Delta$ & Type & s & $\dim V$ & $\dim G/V$ & $\mathfrak g$ & $\mathfrak v$ \\
\hline
\VoganDiag{black}{white}
& $ \varphi_1 $ & no & sCY & $0$ & $4$ & $6$ & $\mathfrak{so}(4,1)$ & $\mathfrak{su}(2) \oplus \mathbf R$ \\
\hline
\VoganDiag{white}{black}
& $ \varphi_2 $ & yes & GT & $-18$ & $4$ & $6$ & $\mathfrak{so}(2,3)$ & $\mathfrak{su}(2) \oplus \mathbf R$ \\
\hline
\end{tabular}
\end{center}

\vspace{1cm}

\renewcommand{\VoganDiag}[2]
{
\begin{tikzpicture}[scale=0.7]
    \draw[thick,fill={#1}] (0 ,0) circle (1mm);
    \node[anchor=south] at (0,1mm) {$1$};
    \node[anchor=north] at (0,-1mm) {$\gamma_1$};
    \draw[thick,fill={#2}] (12mm,0) circle (1mm);
    \node[anchor=south] at (12mm,1mm) {$3$};
    \node[anchor=north] at (12mm,-1mm) {$\gamma_2$};
    \draw[thick] (2mm,0.6mm) -- (10mm,0.6mm);
    \draw[thick] (2mm,0) -- (10mm,0);
    \draw[thick] (2mm,-0.6mm) -- (10mm,-0.6mm);
\end{tikzpicture}
}
\begin{center}
\hspace*{-30mm}
\begin{tabular}{|*9{c|}}
\hline
\Huge $G_2$ & \multicolumn{5}{|c|}{
\begin{minipage}[b]{0.5\columnwidth}%
\smallskip
$\dim \mathfrak g = 14$. One non-compact simple real form denoted by $\mathfrak{g}_{2(2)}=\mathrm{G}$. 
\end{minipage}
} & \multicolumn{3}{l|}{
\begin{minipage}[b]{0.5\columnwidth}
\smallskip
\raggedright Fundamental dominant weights \\ 
$
\begin{array}{lll}
\varphi_1 &= 2 \gamma_1 + \gamma_2 \\
\varphi_2 &= 3 \gamma_1 + 2 \gamma_2 \\ 
\end{array}
$
\end{minipage}} \\
\hline
\hline
Vogan diagram & $\varphi$ & $\varphi \in \Delta$ & Type & s & $\dim V$ & $\dim G/V$ & $\mathfrak g$ & $\mathfrak v$ \\
\hline
\VoganDiag{black}{white}
& $ \varphi_1 $ & yes & sGT & $-30$ & $4$ & $10$ & $\mathfrak{g}_{2(2)}$ & $\mathfrak{su}(2) \oplus \mathbf R$ \\
\hline
\VoganDiag{white}{black}
& $ \varphi_2 $ & yes & sGT & $-10$ & $4$ & $10$ & $\mathfrak{g}_{2(2)}$ & $\mathfrak{su}(2) \oplus \mathbf R$ \\
\hline
\end{tabular}
\end{center}


\topmargin=-50pt
\textheight=620pt

\renewcommand{\VoganDiag}[3]
{
\begin{tikzpicture}[scale=0.7]
    \draw[thick,fill={#1}] (0,0) circle (1mm);
    \node[anchor=south] at (0,1mm) {$1$};
    \node[anchor=north] at (0,-1mm) {$\gamma_1$};
    \draw[thick,fill={#2}] (12mm,0) circle (1mm);
    \node[anchor=south] at (12mm,1mm) {$1$};
    \node[anchor=north] at (12mm,-1mm) {$\gamma_2$};
    \draw[thick,fill={#3}] (24mm,0) circle (1mm);
    \node[anchor=south] at (24mm,1mm) {$1$};
    \node[anchor=north] at (24mm,-1mm) {$\gamma_3$};
    \draw[thick] (2mm,0) -- (10mm,0);
    \draw[thick] (14mm,0) -- (22mm,0);
\end{tikzpicture}
}
\begin{center}
\hspace*{-30mm}
\begin{tabular}{|*9{c|}}
\hline
\Huge $A_3$ & \multicolumn{5}{|c|}{
\begin{minipage}[b]{0.5\columnwidth}%
\smallskip
$\dim \mathfrak g = 15$. Two non-compact simple real forms with trivial automorphism: $\mathfrak{su}(1,3)$, $\mathfrak{su}(2,2)$.
\end{minipage}
} & \multicolumn{3}{l|}{
\begin{minipage}[b]{0.5\columnwidth}
\smallskip
\raggedright Fundamental dominant weights \\ 
$
\begin{array}{lll}
\varphi_1 &= \frac{3}{4}\gamma_1 + \frac{1}{2}\gamma_2 + \frac{1}{4}\gamma_3 \\ 
\varphi_2 &= \frac{1}{2}\gamma_1 + \gamma_2 + \frac{1}{2}\gamma_3 \\ 
\varphi_3 &= \frac{1}{4}\gamma_1 + \frac{1}{2}\gamma_2 + \frac{3}{4}\gamma_3 \\ 
\end{array}
$
\end{minipage}} \\
\hline
\hline
Vogan diagram & $\varphi$ & $\varphi \in \Delta$ & Type & s & $\dim V$ & $\dim G/V$ & $\mathfrak g$ & $\mathfrak v$ \\
\hline
\VoganDiag{black}{white}{white}
& $\varphi_1$
& no & GT & $-24$ & $9$ & $6$ & $\mathfrak{su}(1,3)$ & $\mathfrak{su}(3) \oplus \mathbf R$ \\
\hline
\VoganDiag{white}{black}{white}
& $\varphi_2$
& no & GT & $-32$ & $7$ & $8$ & $\mathfrak{su}(2,2)$ & $\mathfrak{su}(2) \oplus \mathfrak{su}(2) \oplus \mathbf R$ \\
\hline
\VoganDiag{black}{white}{black}
& $\varphi_1+\varphi_3$
& yes & sGT & $-10$ & $5$ & $10$ & $\mathfrak{su}(2,2)$ & $\mathfrak{su}(2) \oplus \mathbf R \oplus \mathbf R$ \\
\hline
\end{tabular}
\end{center}
%

\renewcommand{\VoganDiag}[3]
{
\begin{tikzpicture}[scale=0.7]
    \draw[thick,fill={#1}] (0,0) circle (1mm);
    \node[anchor=south] at (0,1mm) {$1$};
    \node[anchor=north] at (0,-1mm) {$\gamma_1$};
    \draw[thick,fill={#2}] (12mm,0) circle (1mm);
    \node[anchor=south] at (12mm,1mm) {$2$};
    \node[anchor=north] at (12mm,-1mm) {$\gamma_2$};
    \draw[thick,fill={#3}] (24mm,0) circle (1mm);
    \node[anchor=south] at (24mm,1mm) {$2$};
    \node[anchor=north] at (24mm,-1mm) {$\gamma_3$};
    \draw[thick] (2mm,0.4mm) -- (10mm,0.4mm);
    \draw[thick] (2mm,-0.4mm) -- (10mm,-0.4mm);
    \draw[thick] (14mm,0) -- (22mm,0);
\end{tikzpicture}
}
\begin{center}
\hspace*{-30mm}
\begin{tabular}{|*9{c|}}
\hline
\Huge $B_3$ & \multicolumn{5}{|c|}{
\begin{minipage}[b]{0.5\columnwidth}%
\smallskip
$\dim \mathfrak g = 21$. Three non-compact simple real forms: $\mathfrak{so}(6,1)$, $\mathfrak{so}(4,3)$, $\mathfrak{so}(2,5)$.
\end{minipage}
} & \multicolumn{3}{l|}{
\begin{minipage}[b]{0.5\columnwidth}
\smallskip
\raggedright Fundamental dominant weights \\ 
$
\begin{array}{lll}
\varphi_1 &= \frac{3}{2}\gamma_1 + \gamma_2 + \frac{1}{2}\gamma_3 \\ 
\varphi_2 &= 2\gamma_1 + 2\gamma_2 + \gamma_3 \\ 
\varphi_3 &= \gamma_1 + \gamma_2 + \gamma_3 \\ 
\end{array}
$
\end{minipage}} \\
\hline
\hline
Vogan diagram & $\varphi$ & $\varphi \in \Delta$ & Type & s & $\dim V$ & $\dim G/V$ & $\mathfrak g$ & $\mathfrak v$ \\
\hline
\VoganDiag{black}{white}{white}
& $\varphi_1$
& no & sF & $24$ & $9$ & $12$ & $\mathfrak{so}(6,1)$ & $\mathfrak{su}(3) \oplus \mathbf R$ \\
\hline
\VoganDiag{white}{black}{white}
& $\varphi_2$
& yes & sGT & $-28$ & $7$ & $14$ & $\mathfrak{so}(4,3)$ & $\mathfrak{su}(2) \oplus \mathfrak{su}(2) \oplus \mathbf R$ \\
\hline
\VoganDiag{white}{white}{black}
& $\varphi_3$
& yes & GT & $-50$ & $11$ & $10$ & $\mathfrak{so}(2,5)$ & $\mathfrak{so}(5) \oplus \mathbf R$ \\
\hline
\end{tabular}
\end{center}
%

\renewcommand{\VoganDiag}[3]
{
\begin{tikzpicture}[scale=0.7]
    \draw[thick,fill={#1}] (0,0) circle (1mm);
    \node[anchor=south] at (0,1mm) {$1$};
    \node[anchor=north] at (0,-1mm) {$\gamma_1$};
    \draw[thick,fill={#2}] (12mm,0) circle (1mm);
    \node[anchor=south] at (12mm,1mm) {$1$};
    \node[anchor=north] at (12mm,-1mm) {$\gamma_2$};
    \draw[thick,fill={#3}] (24mm,0) circle (1mm);
    \node[anchor=south] at (24mm,1mm) {$2$};
    \node[anchor=north] at (24mm,-1mm) {$\gamma_3$};
    \draw[thick] (2mm,0) -- (10mm,0);
    \draw[thick] (14mm,0.4mm) -- (22mm,0.4mm);
    \draw[thick] (14mm,-0.4mm) -- (22mm,-0.4mm);
\end{tikzpicture}
}
\begin{center}
\hspace*{-30mm}
\begin{tabular}{|*9{c|}}
\hline
\Huge $C_3$ & \multicolumn{5}{|c|}{
\begin{minipage}[b]{0.5\columnwidth}%
\smallskip
$\dim \mathfrak g = 21$. Two non-compact simple real forms: $\mathfrak{sp}(1,2)$, $\mathfrak{sp}(3,\mathbf{R})$.
\end{minipage}
} & \multicolumn{3}{l|}{
\begin{minipage}[b]{0.5\columnwidth}
\smallskip
\raggedright Fundamental dominant weights \\ 
$
\begin{array}{lll}
\varphi_1 &= \gamma_1 + \gamma_2 + \frac{1}{2}\gamma_3 \\ 
\varphi_2 &= \gamma_1 + 2\gamma_2 + \gamma_3 \\ 
\varphi_3 &= \gamma_1 + 2\gamma_2 + \frac{3}{2}\gamma_3 \\ 
\end{array}
$
\end{minipage}} \\
\hline
\hline
Vogan diagram & $\varphi$ & $\varphi \in \Delta$ & Type & s & $\dim V$ & $\dim G/V$ & $\mathfrak g$ & $\mathfrak v$ \\
\hline
\VoganDiag{black}{white}{white}
& $\varphi_1$
& no & sGT & $-20$ & $11$ & $10$ & $\mathfrak{sp}(1,2)$ & $\mathfrak{sp}(2) \oplus \mathbf R$ \\
\hline
\VoganDiag{white}{black}{white}
& $\varphi_2$
& yes & sF & $14$ & $7$ & $14$ & $\mathfrak{sp}(1,2)$ & $\mathfrak{su}(2) \oplus \mathfrak{su}(2) \oplus \mathbf R$ \\
\hline
\VoganDiag{white}{white}{black}
& $\varphi_3$
& no & GT & $-48$ & $9$ & $12$ & $\mathfrak{sp}(3,\mathbf R)$ & $\mathfrak{su}(3) \oplus \mathbf R$ \\
\hline
\VoganDiag{black}{white}{black}
& $\varphi_1+\varphi_3$
& no & sGT & $-16$ & $5$ & $16$ & $\mathfrak{sp}(3,\mathbf{R})$ & $\mathfrak{su}(2) \oplus \mathbf R \oplus \mathbf R$ \\
\hline
\end{tabular}
\end{center}
%

\topmargin=20pt
\textheight=592pt

\renewcommand{\VoganDiag}[4]
{
\begin{tikzpicture}[scale=0.7]
    \draw[thick,fill={#1}] (0,0) circle (1mm);
    \node[anchor=south] at (0,1mm) {$1$};
    \node[anchor=north] at (0,-1mm) {$\gamma_1$};
    \draw[thick,fill={#2}] (12mm,0) circle (1mm);
    \node[anchor=south] at (12mm,1mm) {$1$};
    \node[anchor=north] at (12mm,-1mm) {$\gamma_2$};
    \draw[thick,fill={#3}] (24mm,0) circle (1mm);
    \node[anchor=south] at (24mm,1mm) {$1$};
    \node[anchor=north] at (24mm,-1mm) {$\gamma_3$};
    \draw[thick,fill={#4}] (36mm,0) circle (1mm);
    \node[anchor=south] at (36mm,1mm) {$1$};
    \node[anchor=north] at (36mm,-1mm) {$\gamma_4$};
    \draw[thick] (2mm,0) -- (10mm,0);
    \draw[thick] (14mm,0) -- (22mm,0);
    \draw[thick] (26mm,0) -- (34mm,0);
\end{tikzpicture}
}
\begin{center}
\hspace*{-39mm}
\begin{tabular}{|*9{c|}}
\hline
\Huge $A_4$ & \multicolumn{5}{|c|}{
\begin{minipage}[b]{0.5\columnwidth}%
\smallskip
$\dim \mathfrak g = 24$. Two non-compact simple real forms with trivial automorphism: $\mathfrak{su}(1,4)$, $\mathfrak{su}(2,3)$.
\end{minipage}
} & \multicolumn{3}{l|}{
\begin{minipage}[b]{0.5\columnwidth}
\smallskip
\raggedright Fundamental dominant weights \\ 
$
\begin{array}{lll}
\varphi_1 &= \frac{4}{5}\gamma_1 + \frac{3}{5}\gamma_2 + \frac{2}{5}\gamma_3 + \frac{1}{5}\gamma_4 \\ \vspace{-4mm} \\
\varphi_2 &= \frac{3}{5}\gamma_1 + \frac{6}{5}\gamma_2 + \frac{4}{5}\gamma_3 + \frac{2}{5}\gamma_4 \\ \vspace{-4mm} \\
\varphi_3 &= \frac{2}{5}\gamma_1 + \frac{4}{5}\gamma_2 + \frac{6}{5}\gamma_3 + \frac{3}{5}\gamma_4 \\ \vspace{-4mm} \\
\varphi_4 &= \frac{1}{5}\gamma_1 + \frac{2}{5}\gamma_2 + \frac{3}{5}\gamma_3 + \frac{4}{5}\gamma_4 \\ \vspace{-4mm} 
\end{array}
$
\end{minipage}} \\
\hline
\hline
Vogan diagram & $\varphi$ & $\varphi \in \Delta$ & Type & s & $\dim V$ & $\dim G/V$ & $\mathfrak g$ & $\mathfrak v$ \\
\hline
\VoganDiag{black}{white}{white}{white}
& $\varphi_1$
& no & GT & $-40$ & $16$ & $8$ & $\mathfrak{su}(1,4)$ & $\mathfrak{su}(4) \oplus \mathbf R$ \\
\hline
\VoganDiag{white}{black}{white}{white}
& $\varphi_2$
& no & GT & $-60$ & $12$ & $12$ & $\mathfrak{su}(2,3)$ & $\mathfrak{su}(3) \oplus \mathfrak{su}(2) \oplus \mathbf R$ \\
\hline
\VoganDiag{black}{white}{white}{black}
&$\varphi_1+\varphi_4$
& yes & sGT & $-28$ & $10$ & $14$ & $\mathfrak{su}(2,3)$ & $\mathfrak{su}(3) \oplus \mathbf R \oplus \mathbf R$ \\
\hline
\VoganDiag{white}{black}{black}{white}
& $\varphi_2+\varphi_3$
& no & sF & $16$ & $8$ & $16$ & $\mathfrak{su}(1,4)$ & $\mathfrak{su}(2) \oplus \mathfrak{su}(2) \oplus \mathbf R \oplus \mathbf R$ \\
\hline
\VoganDiag{black}{black}{black}{black}
&  \shortstack{$\sum_{i=1}^4t_i \varphi_i$ \\ for all $t_i>0$}
& no & sCY & $0$ & $4$ & $20$ & $\mathfrak{su}(2,3)$ & $\mathbf R \oplus \mathbf R \oplus \mathbf R \oplus \mathbf R$ \\
\hline
\end{tabular}
\end{center}
%
\vspace{1cm}

\renewcommand{\VoganDiag}[4]
{
\begin{tikzpicture}[scale=0.7]
    \draw[thick,fill={#1}] (0,0) circle (1mm);
    \node[anchor=south] at (0,1mm) {$1$};
    \node[anchor=north] at (0,-1mm) {$\gamma_1$};
    \draw[thick,fill={#2}] (12mm,0) circle (1mm);
    \node[anchor=south] at (12mm,1mm) {$2$};
    \node[anchor=north] at (12mm,-1mm) {$\gamma_2$};
    \draw[thick,fill={#3}] (24mm,0) circle (1mm);
    \node[anchor=south] at (24mm,1mm) {$2$};
    \node[anchor=north] at (24mm,-1mm) {$\gamma_3$};
    \draw[thick,fill={#4}] (36mm,0) circle (1mm);
    \node[anchor=south] at (36mm,1mm) {$2$};
    \node[anchor=north] at (36mm,-1mm) {$\gamma_4$};
    \draw[thick] (2mm,0.4mm) -- (10mm,0.4mm);
    \draw[thick] (2mm,-0.4mm) -- (10mm,-0.4mm);
    \draw[thick] (14mm,0) -- (22mm,0);
    \draw[thick] (26mm,0) -- (34mm,0);
\end{tikzpicture}
}
\begin{center}
\hspace*{-33mm}
\begin{tabular}{|*9{c|}}
\hline
\Huge $B_4$ & \multicolumn{5}{|c|}{
\begin{minipage}[b]{0.5\columnwidth}%
\smallskip
$\dim \mathfrak g = 36$. Four non-compact simple real forms: $\mathfrak{so}(8,1)$, $\mathfrak{so}(6,3)$, $\mathfrak{so}(4,5)$, $\mathfrak{so}(2,7)$.
\end{minipage}
} & \multicolumn{3}{l|}{
\begin{minipage}[b]{0.5\columnwidth}
\smallskip
\raggedright Fundamental dominant weights \\ 
$
\begin{array}{lll}
\varphi_1 &= 2\gamma_1 + \frac{3}{2}\gamma_2 + \gamma_3 + \frac{1}{2}\gamma_4 \\ 
\varphi_2 &= 3\gamma_1 + 3\gamma_2 + 2\gamma_3 + \gamma_4 \\ 
\varphi_3 &= 2\gamma_1 + 2\gamma_2 + 2\gamma_3 + \gamma_4 \\ 
\varphi_4 &= \gamma_1 + \gamma_2 + \gamma_3 + \gamma_4 \\ 
\end{array}
$
\end{minipage}} \\
\hline
\hline
Vogan diagram & $\varphi$ & $\varphi \in \Delta$ & Type & s & $\dim V$ & $\dim G/V$ & $\mathfrak g$ & $\mathfrak v$ \\
\hline
\VoganDiag{black}{white}{white}{white}
& $\varphi_1$
& no & sF & $80$ & $16$ & $20$ & $\mathfrak{so}(8,1)$ & $\mathfrak{su}(4) \oplus \mathbf R$ \\
& & & & & & & &\\
& $\varphi_1+2\varphi_4$
& no & sF & $52$ & $10$ & $26$ & $\mathfrak{so}(8,1)$ & $\mathfrak{su}(3) \oplus \mathbf R \oplus \mathbf R$ \\
\hline
\VoganDiag{white}{black}{white}{white}
& $\varphi_2$
& no & sGT & $-24$ & $12$ & $24$ & $\mathfrak{so}(6,3)$ & $\mathfrak{su}(3) \oplus \mathfrak{su}(2) \oplus \mathbf R$ \\
\hline
\VoganDiag{white}{white}{black}{white}
& $\varphi_3$
& yes & sGT & $-88$ & $14$ & $22$ & $\mathfrak{so}(4,5)$ & $\mathfrak{so}(5) \oplus \mathfrak{su}(2) \oplus \mathbf R$ \\
\hline
\VoganDiag{white}{white}{white}{black}
& $\varphi_4$
& yes & GT & $-98$ & $22$ & $14$ & $\mathfrak{so}(2,7)$ & $\mathfrak{so}(7) \oplus \mathbf R$ \\
\hline
\end{tabular}
\end{center}
%
%
\renewcommand{\VoganDiag}[4]
{
\begin{tikzpicture}[scale=0.7]
    \draw[thick,fill={#1}] (0,0) circle (1mm);
    \node[anchor=south] at (0,1mm) {$1$};
    \node[anchor=north] at (0,-1mm) {$\gamma_1$};
    \draw[thick,fill={#2}] (12mm,0) circle (1mm);
    \node[anchor=south] at (12mm,1mm) {$1$};
    \node[anchor=north] at (12mm,-1mm) {$\gamma_2$};
    \draw[thick,fill={#3}] (24mm,0) circle (1mm);
    \node[anchor=south] at (24mm,1mm) {$1$};
    \node[anchor=north] at (24mm,-1mm) {$\gamma_3$};
    \draw[thick,fill={#4}] (36mm,0) circle (1mm);
    \node[anchor=south] at (36mm,1mm) {$2$};
    \node[anchor=north] at (36mm,-1mm) {$\gamma_4$};
    \draw[thick] (2mm,0) -- (10mm,0);
    \draw[thick] (14mm,0) -- (22mm,0);
    \draw[thick] (26mm,0.4mm) -- (34mm,0.4mm);
    \draw[thick] (26mm,-0.4mm) -- (34mm,-0.4mm);
\end{tikzpicture}
}
\begin{center}
\hspace*{-38mm}
\begin{tabular}{|*9{c|}}
\hline
\Huge $C_4$ & \multicolumn{5}{|c|}{
\begin{minipage}[b]{0.5\columnwidth}%
\smallskip
$\dim \mathfrak g = 36$. Three non-compact simple real forms: $\mathfrak{sp}(1,3)$, $\mathfrak{sp}(2,2)$, $\mathfrak{sp}(4,\mathbf{R})$.
\end{minipage}
} & \multicolumn{3}{l|}{
\begin{minipage}[b]{0.5\columnwidth}
\smallskip
\raggedright Fundamental dominant weights \\ 
$
\begin{array}{lll}
\varphi_1 &= \gamma_1 + \gamma_2 + \gamma_3 + \frac{1}{2}\gamma_4 \\ 
\varphi_2 &= \gamma_1 + 2\gamma_2 + 2\gamma_3 + \gamma_4 \\ 
\varphi_3 &= \gamma_1 + 2\gamma_2 + 3\gamma_3 + \frac{3}{2}\gamma_4 \\ 
\varphi_4 &= \gamma_1 + 2\gamma_2 + 3\gamma_3 + 2\gamma_4 \\ 
\end{array}
$
\end{minipage}} \\

\hline
\hline
Vogan diagram & $\varphi$ & $\varphi \in \Delta$ & Type & s & $\dim V$ & $\dim G/V$ & $\mathfrak g$ & $\mathfrak v$ \\
\hline
\VoganDiag{black}{white}{white}{white}
& $\varphi_1$
& no & sGT & $-56$ & $22$ & $14$ & $\mathfrak{sp}(1,3)$ & $\mathfrak{sp}(3) \oplus \mathbf R$ \\
\hline
\VoganDiag{white}{black}{white}{white}
& $\varphi_2$
& yes & sGT & $-22$ & $14$ & $22$ & $\mathfrak{sp}(2,2)$ & $\mathfrak{su}(2) \oplus \mathfrak{sp}(2) \oplus \mathbf R$ \\
\hline
\VoganDiag{white}{white}{black}{white}
& $\varphi_3$
& no & sF & $48$ & $12$ & $24$ & $\mathfrak{sp}(1,3)$ & $\mathfrak{su}(3) \oplus \mathfrak{su}(2) \oplus \mathbf R$ \\
& & & & & & & &\\
& $3\varphi_1+\varphi_3$
& no & sF & $28$ & $8$ & $28$ & $\mathfrak{sp}(1,3)$ & $\mathfrak{su}(2) \oplus \mathfrak{su}(2) \oplus \mathbf R \oplus \mathbf R$ \\
\hline
\VoganDiag{white}{white}{white}{black}
& $\varphi_4$
& no & GT & $-100$ & $16$ & $20$ & $\mathfrak{sp}(4,\mathbf R)$ & $\mathfrak{su}(4) \oplus \mathbf R$\\
\hline
\end{tabular}
\end{center}
%

\vspace{1cm}

\renewcommand{\VoganDiag}[4]
{
\begin{tikzpicture}[scale=0.7]
    \draw[thick,fill={#1}] (-12mm ,0) circle (1mm);
    \node[anchor=north] at (-12mm,-1mm) {$\gamma_1$};
    \draw[thick,fill={#2}] (0,0) circle (1mm);
    \node[anchor=north] at (0,-1mm) {$\gamma_2$};
    \draw[thick,fill={#3}] (30:12mm) circle (1mm);
    \node[anchor=west] at (30:12mm) {$\gamma_3$};
    \draw[thick,fill={#4}] (-30:12mm) circle (1mm);
    \node[anchor=west] at (-30:12mm) {$\gamma_4$};
    \draw[thick] (30:2mm) -- (30:10mm);
    \draw[thick] (-10mm,0) -- (-2mm,0);
    \draw[thick] (-30:2mm) -- (-30:10mm);
\end{tikzpicture}
}
\begin{center}
\hspace*{-41mm}
\begin{tabular}{|*9{c|}}
\hline
\Huge $D_4$ & \multicolumn{5}{|c|}{
\begin{minipage}[b]{0.5\columnwidth}%
\smallskip
$\dim \mathfrak g = 28$. Two non-compact simple real forms with trivial automorphism: $\mathfrak{so}(2,6)$, $\mathfrak{so}(4,4)$.
\end{minipage}
} & \multicolumn{3}{l|}{
\begin{minipage}[b]{0.5\columnwidth}
\smallskip
\raggedright Fundamental dominant weights \\ \vspace{0.11cm}
$
\begin{array}{lll}
\varphi_1 &= \gamma_1 + \gamma_2 + \frac{1}{2}\gamma_3 + \frac{1}{2}\gamma_4 \\ 
\varphi_2 &= \gamma_1 + 2\gamma_2 + \gamma_3 + \gamma_4 \\ 
\varphi_3 &= \frac{1}{2}\gamma_1 + \gamma_2 + \gamma_3 + \frac{1}{2}\gamma_4 \\ 
\varphi_4 &= \frac{1}{2}\gamma_1 + \gamma_2 + \frac{1}{2}\gamma_3 + \gamma_4 \\ 
\end{array}
$
\end{minipage}} \\
\hline
\hline
Vogan diagram & $\varphi$ & $\varphi \in \Delta$ & Type & s & $\dim V$ & $\dim G/V$ & $\mathfrak g$ & $\mathfrak v$ \\
\hline
\VoganDiag{black}{white}{white}{white}
& $\varphi_1$
& no & GT & $-72$ & $16$ & $12$ & $\mathfrak{so}(2,6)$ & $\mathfrak{su}(4) \oplus \mathbf R$\\
\hline
\VoganDiag{white}{black}{white}{white}
& $\varphi_2$
& yes & sGT & $-54$ & $10$ & $18$ & $\mathfrak{so}(4,4)$ & $\mathfrak{su}(2) \oplus \mathfrak{su}(2) \oplus \mathfrak{su}(2) \oplus \mathbf R$\\
\hline
\VoganDiag{black}{white}{black}{white}
& \shortstack{$t_1 \varphi_1 + t_2 \varphi_3$ 
\\ for all $t_1,t_2>0$}
& no & sCY & $0$ & $10$ & $18$ & $\mathfrak{so}(2,6)$ & $\mathfrak{su}(3) \oplus \mathbf R \oplus \mathbf R$\\
\hline
\VoganDiag{black}{white}{black}{black}
& $\varphi_1+\varphi_3+\varphi_4$
& no & sGT & $-22$ & $6$ & $22$ & $\mathfrak{so}(4,4)$ & $\mathfrak{su}(2) \oplus \mathbf R \oplus \mathbf R \oplus \mathbf R$\\
\hline
\end{tabular}
\end{center}
%
\bigskip
%
%
\renewcommand{\VoganDiag}[4]
{
\begin{tikzpicture}[scale=0.7]
    \draw[thick,fill={#1}] (0,0) circle (1mm);
    \node[anchor=south] at (0,1mm) {$1$};
    \node[anchor=north] at (0,-1mm) {$\gamma_1$};
    \draw[thick,fill={#2}] (12mm,0) circle (1mm);
    \node[anchor=south] at (12mm,1mm) {$1$};
    \node[anchor=north] at (12mm,-1mm) {$\gamma_2$};
    \draw[thick,fill={#3}] (24mm,0) circle (1mm);
    \node[anchor=south] at (24mm,1mm) {$2$};
    \node[anchor=north] at (24mm,-1mm) {$\gamma_3$};
    \draw[thick,fill={#4}] (36mm,0) circle (1mm);
    \node[anchor=south] at (36mm,1mm) {$2$};
    \node[anchor=north] at (36mm,-1mm) {$\gamma_4$};
    \draw[thick] (2mm,0) -- (10mm,0);
    \draw[thick] (14mm,0.4mm) -- (22mm,0.4mm);
    \draw[thick] (14mm,-0.4mm) -- (22mm,-0.4mm);
    \draw[thick] (26mm,0) -- (34mm,0);
\end{tikzpicture}
}
\begin{center}
\hspace*{-36mm}
\begin{tabular}{|*9{c|}}
\hline
\Huge $F_4$ & \multicolumn{5}{|c|}{
\begin{minipage}[b]{0.5\columnwidth}%
\smallskip
$\dim \mathfrak g =  52$.
Two non-compact simple real forms denoted by 
$\mathfrak f_{4(4)} = \mathrm{F\,I}$, $\mathfrak f_{4(-20)} = \mathrm{F\,II}$.
\smallskip
\end{minipage}
} & \multicolumn{3}{l|}{
\begin{minipage}[b]{0.5\columnwidth}
\smallskip
\raggedright Fundamental dominant weights \\ 
$
\begin{array}{lll}
\varphi_1 &=& 2\gamma_1 + 3\gamma_2 + 2\gamma_3 + \gamma_4 \\
\varphi_2 &=& 3\gamma_1 + 6\gamma_2 + 4\gamma_3 + 2\gamma_4 \\
\varphi_3 &=& 4\gamma_1 + 8\gamma_2 + 6\gamma_3 + 3\gamma_4 \\
\varphi_4 &=& 2\gamma_1 + 4\gamma_2 + 3\gamma_3 + 2\gamma_4 \\ 
\end{array}
$
\end{minipage}
} \\
\hline
\hline
Vogan diagram & $\varphi$ & $\varphi \in \Delta$ & Type & s & $\dim V$ & $\dim G/V$ & $\mathfrak g$ & $\mathfrak v$ \\
\hline
\VoganDiag{black}{white}{white}{white}
& $\varphi_1$
& yes & sF & $90$ & $22$ & $30$ & $\mathfrak f_{4(-20)}$ & $\mathfrak{so}(7) \oplus \mathbf R$ \\
\hline
\VoganDiag{white}{black}{white}{white}
& $\varphi_2$
& no & sF & $120$ & $12$ & $40$ & $\mathfrak f_{4(-20)}$ & $\mathfrak{su}(3) \oplus \mathfrak{su}(2) \oplus \mathbf R$ \\
& & & & & & & &\\
& $\varphi_1+\varphi_2$
& no & sF & $84$ & $10$ & $42$ & $\mathfrak f_{4(-20)}$ & $\mathfrak{su}(3) \oplus \mathbf R \oplus \mathbf R$ \\
& & & & & & & &\\
& $\varphi_2+3\varphi_4$
& no & sF & $44$ & $8$ & $44$ & $\mathfrak f_{4(-20)}$ & $\mathfrak{su}(2) \oplus \mathfrak{su}(2) \oplus \mathbf R \oplus \mathbf R$ \\
\hline
\VoganDiag{white}{white}{black}{white}
& $\varphi_3$
& no & sGT & $-40$ & $12$ & $40$ & $\mathfrak f_{4(4)}$ & $\mathfrak{su}(3) \oplus \mathfrak{su}(2) \oplus \mathbf R$ \\
\hline
\VoganDiag{white}{white}{white}{black}
& $\varphi_4$
& yes & sGT & $-180$ & $22$ & $30$ & $\mathfrak f_{4(4)}$ & $\mathfrak{sp}(3) \oplus \mathbf R$ \\
\hline
\VoganDiag{black}{black}{white}{white}
& $\varphi_1+\varphi_2$
& no & sF & $84$ & $10$ & $42$ & $\mathfrak f_{4(-20)}$ & $\mathfrak{su}(3) \oplus \mathbf R \oplus \mathbf R$ \\
\hline
\VoganDiag{black}{white}{white}{black}
& $2\varphi_1+\varphi_4$
& no & sGT & $-40$ & $12$ & $40$ & $\mathfrak f_{4(4)}$ & $\mathfrak{so}(5) \oplus \mathbf R \oplus \mathbf R$ \\
\hline
\end{tabular}
\end{center}

\scriptsize

\renewcommand{\VoganDiag}[6]
{
\begin{tikzpicture}[scale=0.6]
    \draw[thick,fill={#1}] (0,0) circle (1mm);
    \node[anchor=north] at (0,-1mm) {$\gamma_1$};
    \draw[thick,fill={#2}] (12mm,0) circle (1mm);
    \node[anchor=north] at (12mm,-1mm) {$\gamma_2$};
    \draw[thick,fill={#3}] (24mm,0) circle (1mm);
    \node[anchor=north] at (24mm,-1mm) {$\gamma_3$};
    \draw[thick,fill={#4}] (36mm,0) circle (1mm);
    \node[anchor=north] at (36mm,-1mm) {$\gamma_4$};
    \draw[thick,fill={#5}] (48mm,0) circle (1mm);
    \node[anchor=north] at (48mm,-1mm) {$\gamma_5$};
    \draw[thick,fill={#6}] (24mm,12mm) circle (1mm);
    \node[anchor=west] at (25mm,12mm) {$\gamma_6$};
    \draw[thick] (2mm,0) -- (10mm,0);
    \draw[thick] (14mm,0) -- (22mm,0);
    \draw[thick] (26mm,0) -- (34mm,0);
    \draw[thick] (38mm,0) -- (46mm,0);
    \draw[thick] (24mm,2mm) -- (24mm,10mm);
\end{tikzpicture}
}
\begin{center}
\hspace*{-35mm}
\begin{tabular}{|*9{c|}}
\hline
\Huge $E_6$ & \multicolumn{5}{|c|}{
\begin{minipage}[b]{0.5\columnwidth}%
\smallskip
$\dim \mathfrak g = 78$.
Two non-compact simple real forms with trivial automorphism denoted by 
$\mathfrak e_{6(2)}=\mathrm{E\,II}$, $\mathfrak e_{6(-14)} = \mathrm{E\,III}$.
\smallskip
\end{minipage}
} & \multicolumn{3}{l|}{
\begin{minipage}[b]{0.5\columnwidth}
\smallskip
\raggedright Fundamental dominant weights \\ 
$
\begin{array}{lll}
\varphi_1 =& \frac{4}{3}\gamma_1 + \frac{5}{3}\gamma_2 + 2\gamma_3 + \frac{4}{3}\gamma_4 + \frac{2}{3}\gamma_5 + \gamma_6 \\ \vspace{-3mm} \\
\varphi_2 =& \frac{5}{3}\gamma_1 + \frac{10}{3}\gamma_2 + 4\gamma_3 + \frac{8}{3}\gamma_4 + \frac{4}{3}\gamma_5 + 2\gamma_6 \\ \vspace{-3mm} \\
\varphi_3 =& 2\gamma_1 + 4\gamma_2 + 6\gamma_3 + 4\gamma_4 + 2\gamma_5 + 3\gamma_6 \\ \vspace{-3mm} \\
\varphi_4 =& \frac{4}{3}\gamma_1 + \frac{8}{3}\gamma_2 + 4\gamma_3 + \frac{10}{3}\gamma_4 + \frac{5}{3}\gamma_5 + 2\gamma_6 \\ \vspace{-3mm} \\
\varphi_5 =& \frac{2}{3}\gamma_1 + \frac{4}{3}\gamma_2 + 2\gamma_3 + \frac{5}{3}\gamma_4 + \frac{4}{3}\gamma_5 + \gamma_6 \\ \vspace{-3mm} \\
\varphi_6 =& \gamma_1 + 2\gamma_2 + 3\gamma_3 + 2\gamma_4 + \gamma_5 + 2\gamma_6 \\ 
\end{array}
$
\end{minipage}
} \\
\hline
\hline
Vogan diagram & $\varphi$ & $\varphi \in \Delta$ & Type & s & $\dim V$ & $\dim G/V$ & $\mathfrak g$ & $\mathfrak v$ \\
\hline
\VoganDiag{black}{white}{white}{white}{white}{white}
& $\varphi_1$ 
& no & GT & $-384$ & $46$ & $32$ & $\mathfrak e_{6(-14)}$ & $\mathfrak{so}(10)\oplus \mathbf R$ \\
\hline
\VoganDiag{white}{black}{white}{white}{white}{white}
& $\varphi_2$
& no & sGT & $-150$ & $28$ & $50$ & $\mathfrak e_{6(2)}$ & $\mathfrak{su}(5)\oplus \mathfrak{su}(2)\oplus \mathbf R$ \\
\hline
\VoganDiag{white}{white}{black}{white}{white}{white}
& $\varphi_3$
& no & sGT & $-58$ & $20$ & $58$ & $\mathfrak e_{6(2)}$ & $\mathfrak{su}(3)\oplus \mathfrak{su}(3)\oplus \mathfrak{su}(2)\oplus \mathbf R$ \\
\hline
\VoganDiag{white}{white}{white}{white}{white}{black}
& $\varphi_6$
& yes & sGT & $-378$ & $36$ & $42$ & $\mathfrak e_{6(2)}$ & $\mathfrak{su}(6)\oplus \mathbf R$ \\
\hline
\VoganDiag{white}{black}{white}{black}{white}{white}
& $\varphi_2+\varphi_4$
& no & sF & $62$ & $16$ & $62$ & $\mathfrak e_{6(-14)}$ &  $\mathfrak{su}(3)\oplus \mathfrak{su}(2)\oplus \mathfrak{su}(2) \oplus \mathbf R \oplus \mathbf R$ \\
\hline
\VoganDiag{black}{white}{white}{white}{black}{white}
& \shortstack{$t_1\varphi_1 + t_2 \varphi_5$ \\ 
s.t. $t_1,t_2>0$}
& no & sCY & $0$ & $30$ & $48$ & $\mathfrak e_{6(-14)}$ & $\mathfrak{so}(8) \oplus \mathbf R \oplus \mathbf R$ \\
\hline
\end{tabular}
\end{center}

\bigskip
\renewcommand{\VoganDiag}[7]
{
\begin{tikzpicture}[scale=0.6]
    \draw[thick,fill={#1}] (0,0) circle (1mm)
    node[anchor=north] at +(0,-1mm) {$\gamma_1$};
    \draw[thick,fill={#2}] (12mm,0) circle (1mm)
    node[anchor=north] at +(0,-1mm) {$\gamma_2$};
    \draw[thick,fill={#3}] (24mm,0) circle (1mm)
    node[anchor=north] at +(0,-1mm) {$\gamma_3$};
    \draw[thick,fill={#4}] (36mm,0) circle (1mm)
    node[anchor=north] at +(0,-1mm) {$\gamma_4$};
    \draw[thick,fill={#5}] (48mm,0) circle (1mm)
    node[anchor=north] at +(0,-1mm) {$\gamma_5$};
    \draw[thick,fill={#6}] (60mm,0) circle (1mm)
    node[anchor=north] at +(0,-1mm) {$\gamma_6$};
    \draw[thick,fill={#7}] (36mm,12mm) circle (1mm)
    node[anchor=west] at +(0,-1mm) {$\gamma_7$};
    \draw[thick] ++(2mm,0) -- ++(8mm,0)
    ++(4mm,0) -- ++(8mm,0)
    ++(4mm,0) -- ++(8mm,0)
    ++(2mm,2mm) -- ++(0,8mm)
    ++(2mm,-10mm) -- ++(8mm,0)
    ++(4mm,0) -- ++(8mm,0);
\end{tikzpicture}
}

\begin{center}
\hspace*{-38mm}
\begin{tabular}{|*9{c|}}
\hline
\Huge $E_7$ & \multicolumn{5}{|c|}{
\begin{minipage}[b]{0.5\columnwidth}%
$\dim \mathfrak g = 133$.
Three non-compact simple real forms denoted by 
$\mathfrak e_{7(7)}=\mathrm{E\,V}$, $\mathfrak e_{7(-5)} = \mathrm{E\,VI}$, \\ $\mathfrak e_{7(-25)}=\mathrm{E\,VII}$.
\end{minipage}
} & \multicolumn{3}{l|}{
\begin{minipage}[b]{0.5\columnwidth}
\smallskip
\raggedright Fundamental dominant weights \\ 
$
\begin{array}{lll}
\varphi_1 &= \frac{3}{2} \gamma_1 + 2 \gamma_2 + \frac{5}{2} \gamma_3 + 3 \gamma_4 + 2 \gamma_5 + \gamma_6 + \frac{3}{2} \gamma_7 \\ \vspace{-0.2cm} \\
\varphi_2 &= 2 \gamma_1 + 4 \gamma_2 + 5 \gamma_3 + 6 \gamma_4 + 4 \gamma_5 + 2 \gamma_6 + 3 \gamma_7 \\ \vspace{-0.2cm} \\
\varphi_3 &= \frac{5}{2} \gamma_1 + 5 \gamma_2 + \frac{15}{2} \gamma_3 + 9 \gamma_4 + 6 \gamma_5 + 3 \gamma_6 + \frac{9}{2} \gamma_7 \\ \vspace{-0.2cm} \\
\varphi_4 &= 3 \gamma_1 + 6 \gamma_2 + 9 \gamma_3 + 12 \gamma_4 + 8 \gamma_5 + 4 \gamma_6 + 6 \gamma_7 \\ \vspace{-0.2cm} \\
\varphi_5 &= 2 \gamma_1 + 4 \gamma_2 + 6 \gamma_3 + 8 \gamma_4 + 6 \gamma_5 + 3 \gamma_6 + 4 \gamma_7 \\ \vspace{-0.2cm} \\
\varphi_6 &= \gamma_1 + 2 \gamma_2 + 3 \gamma_3 + 4 \gamma_4 + 3 \gamma_5 + 2 \gamma_6 + 2 \gamma_7 \\ \vspace{-0.2cm} \\
\varphi_7 &= \frac{3}{2} \gamma_1 + 3 \gamma_2 + \frac{9}{2} \gamma_3 + 6 \gamma_4 + 4 \gamma_5 + 2 \gamma_6 + \frac{7}{2} \gamma_7 \\ \vspace{-0.23cm} \\
\end{array}
$
\end{minipage}
} \\
%
\hline
\hline
Vogan diagram & $\varphi$ & $\varphi \in \Delta$ & Type & s & $\dim V$ & $\dim G/V$ & $\mathfrak g$ & $\mathfrak v$  \\
\hline
\VoganDiag{black}{white}{white}{white}{white}{white}{white}
& $ \varphi_1 $ & no & GT & $-972$ & $79$ & $54$ & $\mathfrak e_{7(-25)}$ & $\mathfrak e_6 \oplus \mathbf R$\\
\hline
\VoganDiag{white}{black}{white}{white}{white}{white}{white}
& $ \varphi_2 $ & no & sGT & $-252$ & $49$ & $84$ & $\mathfrak e_{7(-5)}$ & $\mathfrak{so}(10) \oplus \mathfrak{su}(2) \oplus \mathbf R$ \\
\hline
\VoganDiag{white}{white}{black}{white}{white}{white}{white}
& $ \varphi_3 $ & no & sGT & $-200$ & $33$ & $100$ & $\mathfrak e_{7(7)}$ & $\mathfrak{su}(5) \oplus \mathfrak{su}(3) \oplus \mathbf R$ \\
\hline
\VoganDiag{white}{white}{white}{black}{white}{white}{white}
& $ \varphi_4 $ & no & sCY & $0$ & $27$ & $106$ & $\mathfrak e_{7(-5)}$ & $\mathfrak{su}(4) \oplus \mathfrak{su}(3) \oplus \mathfrak{su}(2) \oplus \mathbf R$ \\
\hline
\VoganDiag{white}{white}{white}{white}{black}{white}{white}
& $ \varphi_5 $ & no & sGT & $-94$ & $39$ & $94$ & $\mathfrak e_{7(-5)}$ & $\mathfrak{su}(6) \oplus \mathfrak{su}(2) \oplus \mathbf R$\\
\hline
\VoganDiag{white}{white}{white}{white}{white}{black}{white}
& $ \varphi_6 $ & yes & sGT & $-990$ & $67$ & $66$ & $\mathfrak e_{7(-5)}$ & $\mathfrak{so}(12) \oplus \mathbf R$\\
\hline
\VoganDiag{white}{white}{white}{white}{white}{white}{black}
& $ \varphi_7 $ & no & sGT & $-504$ & $49$ & $84$ & $\mathfrak e_{7(7)}$ & $\mathfrak{su}(7) \oplus \mathbf R$\\
\hline
\end{tabular}
\end{center}

\bigskip
\renewcommand{\VoganDiag}[8]
{
\begin{tikzpicture}[scale=0.6]
    \draw[thick,fill={#1}] (0,0) circle (1mm)
    node[anchor=north] at +(0,-1mm) {$\gamma_1$};
    \draw[thick,fill={#2}] (10mm,0) circle (1mm)
    node[anchor=north] at +(0,-1mm) {$\gamma_2$};
    \draw[thick,fill={#3}] (20mm,0) circle (1mm)
    node[anchor=north] at +(0,-1mm) {$\gamma_3$};
    \draw[thick,fill={#4}] (30mm,0) circle (1mm)
    node[anchor=north] at +(0,-1mm) {$\gamma_4$};
    \draw[thick,fill={#5}] (40mm,0) circle (1mm)
    node[anchor=north] at +(0,-1mm) {$\gamma_5$};
    \draw[thick,fill={#6}] (50mm,0) circle (1mm)
    node[anchor=north] at +(0,-1mm) {$\gamma_6$};
    \draw[thick,fill={#7}] (60mm,0) circle (1mm)
    node[anchor=north] at +(0,-1mm) {$\gamma_7$};
    \draw[thick,fill={#8}] (40mm,10mm) circle (1mm)
    node[anchor=west] at +(0,-1mm) {$\gamma_8$};
    \draw[thick] ++(2mm,0) -- ++(6mm,0)
    ++(4mm,0) -- ++(6mm,0)
    ++(4mm,0) -- ++(6mm,0)
    ++(4mm,0) -- ++(6mm,0)
    ++(2mm,2mm) -- ++(0,6mm)
    ++(2mm,-8mm) -- ++(6mm,0)
    ++(4mm,0) -- ++(6mm,0);
\end{tikzpicture}
}
%
\begin{center}
\hspace*{-40mm}
\begin{tabular}{|*9{c|}}
\hline
\Huge $E_8$ & \multicolumn{5}{|c|}{
\begin{minipage}[b]{0.4\columnwidth}%
$\dim \mathfrak g =  248$.
Two non-compact simple real forms denoted by 
$\mathfrak e_{8(8)}=\mathrm{E\,VIII}$, $\mathfrak e_{8(-24)} = \mathrm{E\,IX}$.
\end{minipage}
} & \multicolumn{3}{l|}{
\begin{minipage}[b]{0.6\columnwidth}
\smallskip
\raggedright Fundamental dominant weights \\ 
$
\begin{array}{lll}
\varphi_1 &= 2 \gamma_1 + 3 \gamma_2 + 4 \gamma_3 + 5 \gamma_4 + 6 \gamma_5 + 4 \gamma_6 + 2 \gamma_7 + 3 \gamma_8 \\ \vspace{-3mm} \\
\varphi_2 &= 3 \gamma_1 + 6 \gamma_2 + 8 \gamma_3 + 10 \gamma_4 + 12 \gamma_5 + 8 \gamma_6 + 4 \gamma_7 + 6 \gamma_8 \\ \vspace{-3mm} \\
\varphi_3 &= 4 \gamma_1 + 8 \gamma_2 + 12 \gamma_3 + 15 \gamma_4 + 18 \gamma_5 + 12 \gamma_6 + 6 \gamma_7 + 9 \gamma_8 \\ \vspace{-3mm} \\
\varphi_4 &= 5 \gamma_1 + 10 \gamma_2 + 15 \gamma_3 + 20 \gamma_4 + 24 \gamma_5 + 16 \gamma_6 + 8 \gamma_7 + 12 \gamma_8 \\ \vspace{-3mm} \\
\varphi_5 &= 6 \gamma_1 + 12 \gamma_2 + 18 \gamma_3 + 24 \gamma_4 + 30 \gamma_5 + 20 \gamma_6 + 10 \gamma_7 + 15 \gamma_8 \\ \vspace{-3mm} \\
\varphi_6 &= 4 \gamma_1 + 8 \gamma_2 + 12 \gamma_3 + 16 \gamma_4 + 20 \gamma_5 + 14 \gamma_6 + 7 \gamma_7 + 10 \gamma_8 \\ \vspace{-3mm} \\
\varphi_7 &= 2 \gamma_1 + 4 \gamma_2 + 6 \gamma_3 + 8 \gamma_4 + 10 \gamma_5 + 7 \gamma_6 + 4 \gamma_7 + 5 \gamma_8 \\ \vspace{-3mm} \\
\varphi_8 &= 3 \gamma_1 + 6 \gamma_2 + 9 \gamma_3 + 12 \gamma_4 + 15 \gamma_5 + 10 \gamma_6 + 5 \gamma_7 + 8 \gamma_8 
\end{array}
$

\end{minipage}
} \\


\hline
\hline
Vogan diagram & $\varphi$ & $\varphi \in \Delta$ & Type & s & $\dim V$ & $\dim G/V$ & $\mathfrak g$ & $\mathfrak v$ \\
\hline
\VoganDiag{black}{white}{white}{white}{white}{white}{white}{white}
& $ \varphi_1 $ & yes & sGT & $-3078$ & $134$ & $114$ & $\mathfrak e_{8(-24)}$ & $e_7 \oplus \mathbf R$ \\
\hline
\VoganDiag{white}{black}{white}{white}{white}{white}{white}{white}
& $ \varphi_2 $ & no & sGT & $-166$ & $82$ & $166$ & $\mathfrak e_{8(-24)}$ & $e_6 \oplus \mathfrak{su}(2) \oplus \mathbf R$ \\
\hline
\VoganDiag{white}{white}{black}{white}{white}{white}{white}{white}
& $ \varphi_3 $ & no & sGT & $-388$ & $54$ & $194$ & $\mathfrak e_{8(8)}$ & $\mathfrak{so}(10) \oplus \mathfrak{su}(3) \oplus \mathbf R$ \\
\hline
\VoganDiag{white}{white}{white}{black}{white}{white}{white}{white}
& $ \varphi_4 $ & no & sGT & $-208$ & $40$ & $208$ & $\mathfrak e_{8(8)}$ & $\mathfrak{su}(5) \oplus \mathfrak{su}(4) \oplus \mathbf R$\\
\hline
\VoganDiag{white}{white}{white}{white}{black}{white}{white}{white}
& $ \varphi_5 $ & no & sF & $212$ & $36$ & $212$ & $\mathfrak e_{8(-24)}$ & $\mathfrak{su}(5) \oplus \mathfrak{su}(3) \oplus \mathfrak{su}(2) \oplus \mathbf R$\\
\hline
\VoganDiag{white}{white}{white}{white}{white}{black}{white}{white}
& $ \varphi_6 $ & no & sF & $196$ & $52$ & $196$ & $\mathfrak e_{8(-24)}$ & $\mathfrak{su}(7) \oplus \mathfrak{su}(2) \oplus \mathbf R$\\
\hline
\VoganDiag{black}{white}{white}{white}{white}{black}{white}{white}
& $ \varphi_1  + \varphi_6 $ & no & sGT & $-208$ & $40$ & $208$ & $\mathfrak e_{8(8)}$ & $\mathfrak{su}(6) \oplus \mathfrak{su}(2) \oplus \mathbf R \oplus \mathbf R$\\
\hline
\VoganDiag{white}{white}{white}{white}{white}{white}{black}{white}
& $ \varphi_7 $ & no & sGT & $-1404$ & $92$ & $156$ & $\mathfrak e_{8(8)}$ & $\mathfrak{so}(14) \oplus \mathbf R$\\
\hline
\VoganDiag{white}{white}{white}{white}{white}{white}{white}{black}
& $ \varphi_8 $ & no & sGT & $-552$ & $64$ & $184$ & $\mathfrak e_{8(8)}$ & $\mathfrak{su}(8) \oplus \mathbf R$\\
\hline
\end{tabular}
\end{center}

\pagebreak

\end{document}